\documentclass[11pt]{article}
\usepackage{amssymb,latexsym,graphics,enumerate}
\usepackage{mathrsfs}
\usepackage{amsmath,amsfonts,amsthm,amssymb}
\usepackage{setspace}
\usepackage{Tabbing}
\usepackage{ dsfont }
\usepackage{fancyhdr}
\usepackage{lastpage}
\usepackage{authblk}
\usepackage{extramarks}
\usepackage{chngpage}
\usepackage{soul,color}
\usepackage{soul}
\usepackage{graphicx,float,wrapfig}
\usepackage{tikz}
\usetikzlibrary{matrix,arrows,decorations.pathmorphing}
\usepackage[margin=1.1in]{geometry}                
\usepackage{epstopdf}
\allowdisplaybreaks
\providecommand\abstractname{Abstract}
\def\abstract{}

\renewenvironment{abstract}{%
  \centering\small
  \textbf\abstractname
  \list{}{\leftmargin2cm \rightmargin\leftmargin}
  \item\relax
}{%
  \endlist \par\bigskip
}

\title{ \textsc{Uniform in $N$ Global Well-Posedness of the Time-Dependent Hartree-Fock-Bogoliubov Equations  in $\rr^{1+1}$}}
\author{ \normalsize \textsc{Jacky Jia Wei Chong}}
\affil{\textsc{University of Maryland, College Park}}
\date{}                                           

\newcommand{\Lp}[2]{\left\Vert \, #1 \, \right\Vert_{#2}}

\newcommand{\llp}[1]{ \Vert \, #1 \, \Vert }
\newcommand{\rr}{\mathbb{R}}

\newcommand{\nn}{\mathbb{N}}

\newcommand{\bd}{\partial}
\newcommand{\lapl}{\Delta}

\newcommand*{\bigchi}{\mbox{\Large$\chi$}}

\newcommand{\Tr}{\operatorname{Tr}}
\newcommand{\grad}{\nabla}
\newcommand{\vect}[1]{\mathbf{ #1 }}

\newcommand{\ch}{\operatorname{ch}}
\newcommand{\sh}{\operatorname{sh}}

\newtheorem{thm}{Theorem}[section]
\newtheorem{lem}[thm]{Lemma}

\newtheorem{prop}[thm]{Proposition}
\newtheorem{cor}[thm]{Corollary}
\newtheorem{remark}[thm]{Remark}
\theoremstyle{definition}

\begin{document}

\maketitle
\begin{abstract}We prove the global well-posedness of the time-dependent
Hartree-Fock-Bogoliubov (TDHFB) equations in $\rr^{1+1}$ with two-body interaction
 potential of the form $N^{-1}v_N(x) = N^{\beta-1} v(N^\beta x)$ where $v\geq 0$ is a 
sufficiently regular radial function, i.e. $v \in L^1(\rr)\cap C^\infty(\rr)$. In 
particular, using methods of dispersive PDEs similar to the ones used in \cite{GrillakisMachedon}, we
 are able to show for any scaling parameter $\beta>0$ the TDHFB equations are globally 
well-posed in some Strichartz-type spaces independent of $N$, cf. \cite{BaBreCFSig}.  
\end{abstract}
\section{Introduction}
Let us consider a closed system of $N$ spinless, identical, non-relativistic interacting bosons in $\rr^d$ for $d\leq 3$ 
with pairwise interaction potential $\lambda w$ 
where $\lambda$ is the coupling constant. The evolution of the 
system in the bosonic space $\otimes^N_sL^2(\rr^d)$ is governed by the linear Schr\"odinger equation
\begin{align}
 \frac{1}{i}\frac{\bd}{\bd t} \Psi_N(t, X_N) = \sum^N_{i=1}\lapl_{x_i}\Psi_N(t, X_N)-\lambda\sum_{1\leq i<j \leq N}w(x_i-x_j)\Psi_N(t, X_N)
\end{align}
with $X_N = (x_1, \ldots, x_N) \in \rr^{dN}$. In this article, 
we are interested in the model where both
the kinetic energy and the interaction potential energy are scaled in a similar fashion. In particular, since 
the Hamiltonian
\begin{align}\label{Hamiltonian}
 H_N := \sum^N_{i=1}\lapl_{x_i}-\lambda\sum_{1\leq i<j \leq N}w(x_i-x_j)
\end{align}
scales like $\mathcal{O}(N)+\lambda\mathcal{O}(N^2)$, then the energy of each particle 
is $\mathcal{O}(1)$ provided the coupling constant $\lambda$ is $\mathcal{O}(N^{-1})$, which we called the \emph{mean-field scaling} of (\ref{Hamiltonian}). 
With this scaling, we define the \emph{mean-field limit}\footnote{cf. \cite{frohlich}. A more rigorous definition of the mean-field limit refers to the factorization of the marginal density matrices, for a system of initially uncorrelated particles,  into tensor products of mean fields in trace norm as $N\rightarrow \infty$. cf. Ch 1.11 of \cite{golse}.} to be the singular limit of (\ref{Hamiltonian}) as $N\rightarrow \infty$. Some of the more recent quantitative studies  on the rate of convergence of mean-field limit toward Hartree dynamics can be found in \cite{RS, CLS, CL, kuz}.  
 
To physically motivate the mean-field model, let us consider $N$ particles inside a fixed box\footnote{The box model is used to simplify the exposition. Alternatively, we could have considered $N$ particles in $\rr^d$ subjected to some harmonic trapping potential, i.e. x
\begin{align*}
H_N = \sum^N_{i=1}\{\lapl_{x_i} - v_{\text{ext}}(x_i)\} - \sum_{1 \leq i < j\leq N} w(x_i-x_j)
\end{align*}
where $v_{\text{ext}}$ is small inside the box $[-L, L]$ and large otherwise.} with volume $V =\ell^d$
subjected to either Robin or Neumann boundary conditions. Furthermore, assume the particles interact through a two-body repulsive potential $w$ (with coupling constant $\lambda$ set to 1).
Then the particles will uniformly spread themselves inside the box with an 
average separation distance of $N^{-1/d} \ell$ since the average volume occupied 
by a particle is $N^{-1}\ell^d$. In particular, we are interested in the \emph{dilute gas model}, that is the case when $N^{-1/d}\ell\gg 1$. 
Following a scaling argument, one can show that the dynamics generated by the Hamiltonian (\ref{Hamiltonian}) is
equivalent to the dynamics generated by the rescaled Hamiltonian\footnote{To preserve the dynamics, we will need to rescale the time by a factor of $N^{-2}$.}  
\begin{align}\label{scaled-Hamil2}
 \frac{1}{N^{3-d}}\sum^N_{i=1} \lapl_{y_i} - \frac{1}{N} \sum_{1 \leq i < j \leq N} w_N(y_i-y_j)
\end{align}
where $w_N(y) = N^dw(N y)$ provided we set the length scale of $y_i$ to order 1\footnote{Here we are assuming $x_i$ is on the length scale $\ell\sim N$. }. In the case $d=3$, we see that (\ref{scaled-Hamil2}) gives us a mean-field model for the particles in a unit box with interactions $v_N$.
 Finally, if we take the \emph{dilute limit}, $N^{-1/d}\ell\rightarrow \infty$,
 in the box $\ell^3$, we essentially recover the mean-field limit
 of $N$ weakly interacting particles in the unit box. 
In particular, the 3D mean-field model in the unit box is equivalent to
the strongly interacting dilute gas model in a box. We refer
 the interested reader to \cite{Lewin, lieb, golse} for more in-depth discussions.

Motivated by the above discussion\footnote{It should be noted that the 1D and 2D
 mean-field model can only correspond to the weakly interacting dense gas model.},
 we are lead to consider the mean-field Hamiltonian
\begin{align}\label{mean-field}
 H_{N, \text{mf}} = \sum^N_{i=1} \lapl_{x_i} - \frac{1}{N} \sum_{1 \leq i < j \leq N} v_N(x_i-x_j)
\end{align}
where $v_N(x)=N^{d\beta}v(N^\beta x)$ for $d\leq 3$ and $ v \in C^\infty(\rr^d) \cap L^1(\rr^d)$ which is spherically symmetric. 
The reader should take note of the two scaling processes that are involved in the interactions of
 this mean-field model. Aside from the obvious mean-field scaling, we also have the \emph{short-range scaling} of the 
interaction $v$ given by $v_N$ with a tuning parameter $\beta>0$. 
Let us consider the dynamics generated by the mean-field Hamiltonian and let $\Psi_N$ be the solution to
\begin{align}
 \frac{1}{i}\frac{\bd}{\bd t} \Psi_N = H_{N, \text{mf}}\Psi_N
\end{align}
then by rescaling the solution, i.e. defining $\Phi(\tau, y) = \Psi_N(N^{-2\beta}\tau, N^{-\beta}y)$,
we see the dynamics of the rescaled system is governed by the equation
\begin{align}
 \frac{1}{i}\frac{\bd}{\bd \tau}  \Phi = \sum^N_{i=1} \lapl_{y_i} \Phi - N^{(d-2)\beta-1}\sum_{1 \leq i<j \leq N} v(y_i-y_j)\Phi 
\end{align}
In the instance of $d = 3$, we see, at least heuristically, the appearance of 
a critical scaling when $\beta = 1$, which we called the \emph{Gross-Pitaveskii scaling}. 
Some of the important works done for the case $\beta=1$ in illustrating the 
change in the effective dynamics and the emergence of the scattering length 
can be found in \cite{ErdosSchleinYau, BeOSch, BoCSch}. 
 Moreover, it is heuristically clear that there is no critical scaling when $d = 1, 2$. 
To be more specific, for $d  \leq 2$, the 
coupling constant for the interaction of the rescaled system is inversely proportional 
to the number of particles which means the mean-field scaling is more prominent than the 
short-range scaling effect. Thus, we do not expect to see 
any short scale correlation effects. One of the purposes of this article is to offer a preliminary step to 
a rigorous demonstration 
of the fact that there is no development of short scale correlation structure when $d = 1$ for the effective description by showing the effective equations are well-posed for all $\beta>0$.
The case $d = 2$ for all $\beta>0$ is still open. 

Another reason to consider the entire range of $\beta$ in $\rr^{1+1}$ is inspired by 
the Lieb-Liniger model \cite{lieb63, lieb63-2} which is a 1D model for a system of ultradcold 
Bose particles inside the torus endowed with  a pairwise interaction given by the
 repulsive $\delta$-function, i.e.  the Lieb-Liniger Hamiltonian for the $N$-particle
 Bose gas, in appropriate units, is
\begin{align}\label{LLH}
H_N = -\sum^N_{i=1} \frac{\bd^2}{\bd x_i^2} + 2c \sum_{1 \leq i<j \leq N} \delta(x_i-x_j)
\end{align}
where $c\geq 0$ denotes the repulsion strength.  
More specifically, one can view the Lieb-Liniger model on $\rr$ as a 
heuristic endpoint case of our analysis of the dynamics generated by (\ref{mean-field}) in the weak-coupling limit regime, $c\rightarrow 0$. 

Our interest in the model is twofold.
 From a phsyics point of view, the model has an important feature of being exactly solvable
 in the ground state with computable spectrum. Moreover, the recent advancement
 in the techniques of trapping and cooling atoms has opened up a variety of 
possible experimental studies for ultracold Bose gases that are effectively one-dimensional; for a comprehensive survey on the subject, we refer the reader to
\cite{bloch}.  
 Hence a firm mathematical understanding of the dynamics generated by the Hamiltonian (\ref{LLH})
is an indispensable theoretical tool to suggest further experimental investigation 
of certain 1D properties for ultracold Bose gases. In particular, 
an effective description of the dynamics generated by the Lieb-Liniger model would 
provide a simplified way to analyze the dynamics of these effectively one-dimensional 
Bose gases. From a mathematical perspective, the Lieb-Liniger model on $\rr$ 
is the simplest instance of a many-body quantum mechanical model with 
interaction given by the $\delta$-potential. Up to date, there is no rigorous results
 on the effective description of the evolution of any quantum system 
with $\delta$-interaction.

In this article, we are interested in studying the well-posedness of the time-dependent Hartree-Fock-Bogoliubov (TDHFB) 
equations which, in 3D, describes the quantum fluctations of the Bose field around a Bose-Einstein consendate in the ``absolute-zero temperature'' model. 
These equations were first rigorously derived as Euler-Lagrange equations 
in \cite{GrillakisMachedon2015}, which in turn is based on earlier works by the same authors with collaborator in \cite{GMachMarg, GMachMarg2}. 
Later, in \cite{GrillakisMachedon}, Grillakis and Machedon rederived the TDHBF equations
as evolution equations for the Fock space marginal densities subjected to some reduced dynamics and used techniques from dispersive PDEs to study the local well-posedness of the coupled system\footnote{cf. \cite{BaBreCFSig}}. 
The TDHFB equations are 
\begin{subequations}\label{TDHFB}
\begin{align}
 \frac{1}{i}\bd_t\varphi_t =&\ \lapl\varphi_t - (v_N\ast \rho_{\Gamma_t})\varphi_t -\kappa(\Gamma_t^{\varphi_t}) \varphi_t- \kappa(\Lambda_t^{\varphi_t})\overline{\varphi}_t \\
 \frac{1}{i}\bd_t\Gamma_t =&\ [\lapl-v_N\ast \rho_{\Gamma_t}, \Gamma_t]-[\kappa(\Gamma_t^{\varphi_t}), \Gamma_t]-[\kappa(|\varphi_t\rangle\langle\varphi_t|), \Gamma^{\varphi_t}_t]\\
 &\ - \kappa(\Lambda_t^{\varphi_t})\Lambda^\ast_t+\Lambda_t^{\varphi_t}\kappa(\Lambda_t)^\ast-\kappa(\varphi_t^{\otimes 2})(\Lambda_t^{\varphi_t} )^\ast+\varphi_t^{\otimes 2}\kappa(\Lambda^{\varphi_t}_t)^\ast \nonumber\\
 \frac{1}{i}\bd_t \Lambda_t =&\ \{\lapl - v_N\ast\rho_{\Gamma_t}, \Lambda_t\}-\frac{1}{N}\kappa(\Lambda_t)-\{\kappa(\Lambda_t^{\varphi_t}), \bar \Gamma_t\}\\
 &\ -\{\kappa(\overline{\Gamma}_t), \Lambda_t^{\varphi_t}\}-\{\kappa(\overline{\Gamma}_t^{\varphi_t}), \varphi^{\otimes 2}_t\}-\{\kappa(\Lambda_t^{\varphi_t}), \varphi_t\otimes\overline{\varphi}_t\} \nonumber
\end{align}
\end{subequations}
where $\{A, B\} = AB^T+BA^T, \Gamma^\varphi := \Gamma -|\varphi\rangle\langle\varphi| $ and $\Lambda^\varphi := \Lambda - \varphi\otimes\varphi$ and $\kappa:\alpha \rightarrow \kappa(\alpha)$ has the integral kernel given by
\begin{align*}
[\kappa(\alpha)](x, y) = v_N(x-y)\alpha(x, y).
\end{align*}
A more explicit form of the equations in terms of the kernels can be found in \textsection \ref{Proof-of-main}.
Independently and in a different frame work,  Bach, Breteaux, Chen, Fr\" ohlich, and 
Sigal rigorously derived and studied the well-posedness of a set of coupled equations closely related to the above equations in \cite{BaBreCFSig}. More precisely, 
the triplet $(\phi_t, \gamma_t, \sigma_t)$, introduced in \cite{BaBreCFSig}, corresponds to
\begin{subequations}
\begin{align}
\phi_t =&\ \sqrt{N}\varphi_t,\\
  \gamma_t=&\ N(\overline{\Gamma_t-|\varphi_t\rangle\langle\varphi_t|}) = \overline{\sh(k)}\circ \sh(k),\\
   \sigma_t=&\ N(\Lambda_t-\varphi_t\otimes\varphi_t) =  \sh(k)\circ\ch(k)
\end{align}
\end{subequations}
when written in the notations of \cite{GrillakisMachedon2015,GrillakisMachedon}. See \textsection \ref{notation} for more details on the notation. 

More recently, Benedikter, Sok and Solovej use the reformulated Dirac-Frenkel variational principle in the space of reduced density matrices to geometrically approximate\footnote{They were able to show that the Dirac-Frenkel variational principle implies the quasifree reduction principle which was used in \cite{BaBreCFSig}.} the dynamics of  both the bosonic and fermionic many-body systems in \cite{BSS}. Using the variational principle, they provide a rigorous derivation of both the TDHFB equations and the Bogoliubov-de-Gennes equations, also known as the fermionic TDHFB equations, and show the equations are optimal approximations of the many-body dynamics when restricted to the manifold of quasi-free states\footnote{See \textsection 10 in \cite{Solo} for a definition of quasi-free states.}. We also refer the interested reader to \cite{HLLS} for a study of the pseudo-relativistic version of the Bogoliubov-de-Gennes equations. 

\subsection*{Acknowledgement} The author would like to take this opportunity to express his deepest gratitude toward his two advisors M. Grillakis and M. Machedon for all the time and energy they have spent on the author. Moreover,  the author would also like to thank the editors and referees for providing valuable feedbacks. In particular, one of the referees has the author's sincere appreciation for submitting a very detailed and thoughtful review of the article, which significantly contributed to improving the presentation quality of the paper.

\section{Notations and Main Statement}\label{notation}
Let us indicate some of the notations adopted by the article.\

\noindent\textbf{Notations}. Following \cite{GrillakisMachedon}, we use the notations
\begin{align*}
\vect{S}_\pm := \frac{1}{i}\frac{\bd}{\bd t} - \lapl_{x} +\lapl_{y} 
\ \ \ \text{ and } \ \ \ \vect{S} := \frac{1}{i}\frac{\bd}{\bd t} - \lapl_{x} -\lapl_{y}
\end{align*}
to denote the two Schr\"odinger-type differential operators. Moreover, unless specified, $x, y \in \rr$, which means
$\lapl_x = \bd_{xx}$ and, similarly, $\lapl_y = \bd_{yy}$. The two types of semilinear equations, corresponding to the above operators, considered are
 the inhomogeneous von-Neumann Schr\" odinger equation
\begin{align}\label{nl-gamma}
\vect{S}_\pm \Gamma = F
\end{align}
and the inhomogeneous Schr\"odinger equation 
\begin{align}\label{nl-lambda}
\left(\vect{S}+\frac{1}{N}v_N(x-y)\right)\Lambda = F
\end{align}
where $v_N(x) = N^\beta v(N^\beta x)$
and $v \in L^1(\rr)\cap C^\infty(\rr)$. 
\begin{remark}
We assume $v$ is non-negative and even in our presentation since our prime interest is in studying $v_N\rightarrow c\delta$. However, it should be noted that $v$ can be asymmetric and negative when we study the local well-posedness of the TDHFB; whether these facts have interesting physical consequences will not be explored in this article.  
\end{remark}

Next, let us define the space for the initial data. For every $s>0$, we define the space
\begin{align*}
 \mathcal{X}^s=\{(\varphi, \Gamma, \Lambda) \in H^s\times H^s_\text{Herm}\times H^s_{\text{sym}}\}
\end{align*}
with $H^s$ being the Sobolev space $H^s(\rr), H^s_{\text{Herm}}$ the Sobolev space 
$H^s(\rr^2)$ restricted to functions $\Gamma$ such that $\Gamma(x, y) =\overline{\Gamma(y, x)}$, and $H^s_{\text{sym}}$ the Sobolev space
space $H^s(\rr^2)$ restricted to functions $\Lambda$ such that $\Lambda(x, y) = \Lambda(y, x)$. More specifically, $\mathcal{X}^s$ is endowed with
the norm
\begin{align*}
 \llp{(\varphi, \Gamma, \Lambda)}_{\mathcal{X}^s}:=&\ \llp{\langle\grad_x\rangle^s\phi}_{L^2(\rr)}+\llp{(\langle\grad_x\rangle^2\otimes 1+1\otimes \langle \grad_y\rangle^2)^{s/2}\Gamma}_{L^2(\rr^2)}\\
&\ +\llp{(\langle\grad_x\rangle^2\otimes 1+1\otimes \langle \grad_y\rangle^2)^{s/2}\Lambda}_{L^2(\rr^2)}.
\end{align*}
When the context is clear, we use the symbol $\langle \grad_{x, y}\rangle^s$ in place of $(\langle\grad_x\rangle^2\otimes 1+1\otimes \langle \grad_y\rangle^2)^{s/2}$. Furthermore, we study the local well-posedness of our equations in some Strichartz spaces, which are mixed $L^p$ spaces endowed with the norm
\begin{align*}
\llp{u}_{L^q[0, T]L^r(dx)L^s(dy)} := \left(\int^T_0 dt\ \left(\int dx\ \llp{u(t, x, \cdot}_{L^s(dy)}^r\right)^{q/r}\right)^{1/q}
\end{align*}
where the triplet $(q, r, s)$ satisfies some Strichartz admissible conditions, which will be made clear in the following sections. We also adopt the equivalent notation $L^q(dt)L^r(dx)L^s(dy)$, with the implicit assumption that it depends on $T$, in place of $L^q[0, T]L^r(dx)L^s(dy)$.
 
The hyperbolic trigonometric integral operators introduced in \textsection 1 are defined as follows
\begin{align*}
 \sh(k) :=&\ k + \frac{1}{3!}k\circ\bar k \circ k +\frac{1}{5!}k\circ \bar k \circ k\circ \bar k \circ k+\ldots \\
 \ch(k) :=&\ \delta+p(k):=\delta + \frac{1}{2!} \bar k\circ k+\frac{1}{4!}\bar k\circ k \circ\bar k\circ k + \ldots
\end{align*}
where $\circ$ indicates composition of operators. The symmetric kernel of $k$, $k(t, x, y)=k(t, y, x)$, 
is called the pair excitation function. The following are some useful trigonometric identities
\begin{subequations}
\begin{align}
 &\sh(2k)=\ 2\sh(k)\circ\ch(k),\ \ \ch(2k)=\ \delta + 2\overline{\sh(k)}\circ \sh(k) \label{trig-1}\\
 &\ch(k)\circ\ch(k)-\overline{\sh(k)}\circ \sh(k) = \delta. \label{trig-2}
\end{align}
\end{subequations}
Lastly, we use the usual conventional
notation
\begin{align*}
 \rho_{\Gamma}(t, x) :=\Gamma(t, x, x)
\end{align*}
to define the restriction of $\Gamma$ to the diagonal of the plane. 
\begin{remark}\label{HS-remark}
We adopt the usual convention of identifying the collection of 
Hilbert-Schmidt integral operators on $L^2(\rr^d)$, denoted by $\mathcal{L}^2$, with their integral kernels
in $L^2(\rr^d\times \rr^d)$.  
\end{remark}

\noindent\textbf{Main Statement and Structure}. Let us state the main results of the article 
\begin{thm}[Uniform in $N$ Local Well-Posedness of the TDHFB  in $\rr^{1+1}$]\label{main}
Suppose $\beta>0$ and $R>0$. Then there exist $T=T(\beta, R)>0$, $\sigma=\sigma(\beta)$, both independent of $N$, and a corresponding spacetime function space $X_T$, depending only on $T$ and $\sigma$,  such that for any
given 
\begin{align*}
(\varphi_0, \Gamma_0, \Lambda_0) \in \{ (\varphi, \Gamma, \Lambda) \in \mathcal{X}^\sigma\mid \llp{(\varphi, \Gamma, \Lambda)}_{\mathcal{X}^\sigma} < R\},
\end{align*} 
there exists a unique solution to 
the TDHFB equations (\ref{TDHFB})  with initial data $(\varphi_0, \Gamma_0, \Lambda_0)$ satisfying $(\varphi_t, \Gamma_t, \Lambda_t) \in C([0, T]\rightarrow \mathcal{X}^\sigma)\cap X_T$.
\end{thm}
\begin{remark}
The proof is based on Picard-Lindel\"of theorem or sometimes known as the Banach fixed-point method.
We refer the reader to \textsection \ref{Proof-of-main} for the definition of the function space $X_T$ and Theorem \ref{main thm} for the a-priori estimates involved in the proof of Theorem \ref{main}. 
\end{remark}
\begin{remark}
Given $\beta>0$, we will later see that the choice of $\sigma$ must satisfy the conditions $1-\beta \sigma>0$ and $0<\sigma<\frac{1}{2}$; see Remark \ref{sigma-n} and Remark \ref{beta}.  Informally, this means when $\beta$ is large we can only have uniform control of low Sobolev norms. Ideally, we would like
 to choose $\sigma = 0$, but the nonlinearity requires us to choose $\sigma>0$; see Remark \ref{remark-1}.
Hence an interesting point to observe is the competition between  large $\beta>0$, which requires low regularity of the initial condition, and the non-linearity, which requires some regularity.    
\end{remark}

\begin{remark}
The choice of the Banach space $X_T$ is sufficient, maybe necessary, for our analysis of the TDHFB equations. Heuristically, the space $X_T$  is an intersection of Strichartz spaces, which capture evolution due to the Schr\" odinger-type operators, plus a trace-type space, which captures the interactions coming from the nonlinearity of the coupled equations.
\end{remark}
\begin{cor}[Uniform in $N$ Global Well-Posedness of the TDHFB  in $\rr^{1+1}$]  
Suppose $\beta>0$ and $R>0$. Then for any 
\begin{align*}
(\varphi_0, \Gamma_0, \Lambda_0) \in \{ (\varphi, \Gamma, \Lambda) \in \mathcal{X}^\sigma\mid \llp{(\varphi, \Gamma, \Lambda)}_{\mathcal{X}^\sigma}+\llp{(\grad_x\varphi, \grad_{x, y}\Gamma, \grad_{x, y}\Lambda)}_{\mathcal{X}^\sigma} < R\},
\end{align*} 
the corresponding local solution to  the TDHFB equations (\ref{TDHFB}) given by Theorem \ref{main} extends globally with $(\varphi_t, \Gamma_t, \Lambda_t) \in C([0, \infty)\rightarrow \mathcal{X}^\sigma)\cap X_{\infty, \text{loc}}$ (See \textsection 8 for definition of $X_{\infty, \text{loc}}$). 
\end{cor}

\begin{remark}
To prove the global well-posedness it suffices to prove that the following estimates
\begin{align*}
 &\llp{\langle \grad_x\rangle^{\sigma}\varphi(t, \cdot)}_{L^2(dx)}\lesssim 1\\
 &\llp{\langle \grad_{x, y}\rangle^{\sigma}\Gamma(t, \cdot)}_{L^2(dxdy)} \lesssim 1\\
 &\llp{\langle \grad_{x, y}\rangle^{\sigma}\Lambda(t, \cdot)}_{L^2(dxdy)} \lesssim 1
\end{align*}
hold uniformly in $t$ and $N$, which is a consequence of the conservation 
laws proved in \cite{GrillakisMachedon2015}. See  \textsection \ref{Proof-of-main}.3. 
\end{remark}

\begin{remark}
Our result does not require the condition $V^2\leq C(I-\lapl)$ which is a standard assumption used to treat the multiplicative operator $V$ as a perturbation of the non-interacting case. More precisely, since we are working with $V(x) = N^{\beta-1}v(N^\beta x)$, then we see that
\begin{align*}
N^{\beta-2}\int dx\ |v(x)|^2|f(N^{-\beta}x)|^2=\llp{Vf}_{L^2(\rr)}^2 \lesssim \llp{f'}_{L^2(\rr)}^2 + \llp{f}_{L^2(\rr)}^2
\end{align*}  
can only be true uniformly in $N$ provided $\beta<2$. Nevertheless, in the one dimensional setting, $V$
can still be considered as a perturbation even without the condition.
\end{remark}

Now let us explain a bit the structure of the paper. 
In \textsection \ref{sec-3} and \textsection \ref{sec-4}, we develop estimates that are essential for closing
the iteration scheme of the $\Gamma$ equation. The main results of those two sections necessary for the proof of Theorem \ref{main thm} are Proposition \ref{Inhomog-eq}, Proposition \ref{inhomog-eq2} and Propostion \ref{gamma-stric}. Likewise, from \textsection 6 and \textsection 7, we will need Proposition \ref{trace-energy-lam}, Corollary \ref{main-cor}, Proposition \ref{strichartz3}  and Remark \ref{beta} to close
the estimate for the $\Lambda$ equation. Finally, in \textsection\ref{Proof-of-main} we prove a-priori estimates that are necessary for us to establish the local well-posedness theory for the TDHFB equations then extend the result to a global well-posedness result under further assumption on the initial data.

\section{Estimates for the Homogeneous $\Gamma$ Equation}\label{sec-3}

The main purpose of this section is to prove (\ref{homogeneous-rho-est2}) for the von-Neumann Schr\"odinger equation
\begin{align}\label{vNS-eq}
\frac{1}{i}\frac{\bd}{\bd t}\Gamma +[-\lapl, \Gamma]= 0
\end{align}
for arbitrarily smooth initial condition $\Gamma(0, x, y) = \Gamma_0(x, y)$. The two key ingredients involved in the proof of Corollary \ref{homogeneous-rho-est} are the collapsing estimate and the sharp trace theorem\footnote{Here, sharp trace theorem refers to the statement: for any hyperplane $\Sigma\subset\rr^d$ and $s>1/2$, the trace operator $T:H^{s}(\rr^d)\rightarrow H^{s-1/2}(\Sigma)$ is bounded, i.e. $\llp{Tf}_{H^{s-1/2}(\Sigma)} \lesssim \llp{f}_{H^s(\rr^d)}$.}.

Let us adopt the following convention for our spacetime Fourier transform: the spacetime Fourier transform of a Schwartz function $f \in \mathcal{S}(\rr\times\rr^d)$, denoted by $\widetilde{f}$, is defined to be
 \begin{align}
 \widetilde{f}(\tau, \xi) =\int dtdx\ e^{-i(\tau t+\xi\cdot x)} f(t, x).
 \end{align}
 Likewise, the Fourier transform $\hat f$ of some function $f\in \mathcal{S}(\rr^d)$ is defined by
 \begin{align}
 \hat f (\xi) = \int dx\ e^{-i\xi \cdot x} f(x)
 \end{align}
 with corresponding inversion formula
 \begin{align}
 f(x) = \frac{1}{(2\pi)^d} \int d\xi\ e^{i\xi \cdot x} \hat f(\xi).
 \end{align}
\begin{remark}
The reader should be aware of our attempt to keep track of the values of the fractional derivatives in this section. Keeping a record of these values allows us to show that the mapping used when implementing the fixed-point argument  is indeed a self map.
\end{remark}

Now, using the spacetime Fourier transform, we can establish the following collapsing estimate for the solution to (\ref{vNS-eq}).

\begin{prop}[Collapsing Estimate]\label{collapsing-est}
 Suppose $\Gamma$ is a solution to $\vect{S}_{\pm}\Gamma =0$, then
\begin{align}\label{col-est}
 \llp{\grad_x^{1/2}\rho_\Gamma(t, x)}_{L^2(dtdx)} \lesssim \llp{\Gamma_0}_{L^2(dxdy)}.
\end{align}
\end{prop}
\begin{proof}
Taking the spacetime Fourier transform of $\Gamma$ yields
\begin{align*}
\widetilde{\rho_\Gamma(t, x)} =&\ \int dtdx\ e^{-i\tau t-i\xi\cdot x}\rho_\Gamma(t, x) 
=\ \int dtdxdy\ e^{-i\tau t-i\xi\cdot x}\delta(x-y)\Gamma(t, x, y)\\
=&\ \frac{1}{2\pi}\int d\eta dt\ e^{-i\tau t}\widehat{\Gamma}(t, \xi-\eta, \eta)
=\ \frac{1}{2\pi}\int d\eta dt\ e^{it (-\tau-|\xi-\eta|^2+|\eta|^2)}\widehat{\Gamma_0}(\xi-\eta, \eta)\\
=&\ \frac{1}{2\pi}\int d\eta\ \delta(\tau +|\xi-\eta|^2-|\eta|^2)\widehat{\Gamma_0}(\xi-\eta, \eta)\\
=&\ \frac{1}{4\pi |\xi|}\widehat{\Gamma_0}\left(\frac{\xi^2-\tau}{2\xi}, \frac{\xi^2+\tau}{2\xi}\right).
\end{align*}
Taking the $L_{\tau, \xi}^2(\rr\times\rr)$ norm of $\widetilde{\grad_x^{1/2}\rho_\Gamma}$ and applying Cauchy-Schwarz gives us the estimate
\begin{align*}
\int d\tau d\xi\ |\widetilde{\grad^{1/2}_x\rho_\Gamma(t, x)}(\tau, \xi)|^2 \lesssim \llp{\Gamma_0}_{L^2(dxdy)}^2
\end{align*}
since
\begin{align*}
\sup_{\tau, |\xi|}\int d\eta\ \delta(\tau+|\xi-\eta|^2-|\eta|^2) |\xi|  \lesssim 1.
\end{align*}
\end{proof}

Utilizing the above collapsing estimate, we prove a couple perturbed version 
of the collapsing estimate which will be crucial for our article. 
\begin{lem}
 Suppose $\Gamma$ is a solution to $\vect{S}_\pm \Gamma = 0$. Then for any $\varepsilon>0$ we have the estimate
\begin{align}\label{sharp-trace}
 \llp{\grad^\varepsilon_x\rho_\Gamma(t, x)}_{L^\infty(dt)L^2(dx)} \lesssim 
\llp{\grad^{1/2+\varepsilon}_{x, y}\Gamma_0}_{L^2(dxdy)}.
\end{align}
\end{lem}
\begin{proof}
 For any fixed $t$, it follows from the sharp trace theorem and the conservation of mass we have that
\begin{align*}
 \llp{\grad^\varepsilon_x\rho_\Gamma(t, x)}_{L^2(dx)} \lesssim 
\llp{\grad^{1/2+\varepsilon}_{x, y}\Gamma_0}_{L^2(dxdy)}. 
\end{align*}
\end{proof}

\begin{prop}
  Suppose $\Gamma$ is a solution to $\vect{S}_\pm \Gamma = 0$. Then for any $0<\varepsilon<\frac{1}{2}$ and   $0<\varepsilon'<\frac{1}{2}-\varepsilon$ there exist $q=q(\varepsilon)$ and $\alpha = \alpha(\varepsilon)$ such that the following estimate holds
\begin{align}
 \llp{\grad^{\frac{1}{2}-\varepsilon'}_x\rho_\Gamma(t, x)}_{L^q(dt)L^2(dx)} \lesssim 
\llp{\grad^{\alpha}_{x, y}\Gamma_0}_{L^2(dxdy)}.
\end{align}
\end{prop}

\begin{proof}
 Interpolating\footnote{c.f. Chapter V \textsection 4 in \cite{Stein}. } estimates  (\ref{col-est}) and (\ref{sharp-trace}), we obtain the estimate 
\begin{align*}
 \llp{\grad^{\frac{1}{2}-\varepsilon'}_x\rho_\Gamma}_{L^q(dt)L^2(dx)} \lesssim 
\llp{\grad_{x,y}^\alpha\Gamma_0}_{L^2(dxdy)}
\end{align*}
 with $\alpha$  given by
\begin{align*}
 \alpha = \left(\frac{\frac{1}{2}+\varepsilon}{\frac{1}{2}-\varepsilon}\right)\varepsilon'.
\end{align*}
Moreover, checking the arithmetic, we see that
\begin{align*}
 q = \frac{1-2\varepsilon}{\frac{1}{2}-\varepsilon'-\varepsilon}\geq 2
\end{align*}
since $\varepsilon'<\frac{1}{2}-\varepsilon$. 
\end{proof}

\begin{cor}\label{homogeneous-rho-est}
Suppose $\Gamma$ is a solution to $\vect{S}_\pm \Gamma = 0$. Then for any $0<\varepsilon<1/2$ there exists $q = q(\varepsilon)$ such that the following estimate holds
\begin{align}\label{homogeneous-rho-est2}
 \llp{\grad^{\frac{1}{2}-\varepsilon}_x\rho_\Gamma(t, x)}_{L^q(dt)L^2(dx)} \lesssim 
\llp{\grad^{(\frac{3}{2}-\varepsilon)\varepsilon}_{x, y}\Gamma_0}_{L^2(dxdy)}.
\end{align}
\end{cor}

\begin{proof}
Fix $\varepsilon$. Choose $\delta$ to be
\begin{align*}
   \delta = \frac{1-2\varepsilon}{2(5-2\varepsilon)}  \ \ \implies \ \ \frac{\frac{1}{2}+\delta}{\frac{1}{2}-\delta} = \frac{3}{2}-\varepsilon.
\end{align*}
To avoid confusion, the reader should note that $\varepsilon$ and $\delta$ here correspond to $\varepsilon'$ and $\varepsilon$ in Proposition 3.4.  Hence by the previous proposition, there exists $q(\varepsilon)$ given by
\begin{align*}
 q = \frac{2}{(2-\varepsilon)(\frac{1}{2}-\varepsilon)}
\end{align*}
such that estimate (\ref{homogeneous-rho-est2}) holds. 
\end{proof}

\begin{remark}\label{sigma-n}
For convenience, we shall henceforth denote the quantity $(\frac{3}{2}-\varepsilon)\varepsilon$ by $\sigma$.
\end{remark}

\begin{remark}\label{remark-1}
  Heuristically, we want the estimate
\begin{align*}
 \llp{\rho_\Gamma(t,x)}_{L^2(dt)L^\infty(dx)}
\lesssim \llp{\nabla^{1/2}_x\rho_\Gamma(t, x)}_{L^2(dtdx)}\lesssim 
\llp{\Gamma_0}_{L^2(dxdy)}
\end{align*}
but the estimate is a false endpoint of the Gagliardo-Nirenberg estimate. However, by using
the above corollary and the fact that we are working on a finite interval $[0, T]$, we get
that
\begin{align*}
\llp{\rho_\Gamma(t, x)}_{L^2(dt)L^p(dx)}\lesssim&\ \llp{\grad^{\frac{1}{2}-\varepsilon}_x
\rho_\Gamma(t, x)}_{L^2(dtdx)}\\
\lesssim&\ 
T^\text{some power}\llp{\grad^{\frac{1}{2}-\varepsilon}_x\rho_\Gamma(t, x)}_{L^q(dt)L^2(dx)}\\
\lesssim&\ T^\text{some power}\llp{\grad^{\sigma}_{x,y}
\Gamma_0}_{L^2(dxdy)}.
\end{align*}
We will elaborate more on this point in the next section. 
\end{remark}

Next, let us establish the homogeneous Strichartz estimate for the linear operator $\vect{S}_\pm$. 
\begin{prop}[Non-Endpoint Strichartz] \label{strichartz1}
Suppose $\Gamma$ is a solution to $\vect{S}_\pm\Gamma = 0$ with initial condition $\Gamma_0$ and $(k, \ell)$ is an admissible pair, i.e.
\begin{align}\label{ad-pair}
\frac{2}{k}+\frac{1}{\ell} = \frac{1}{2}
\end{align}
where $(k, \ell) \in (2, \infty]\times [2, \infty]$. Then it follows
\begin{align}
\llp{e^{it(\lapl_x-\lapl_y)}\Gamma_0}_{L^k(dt)L^\ell(dx)L^2(dy)}\lesssim \llp{\Gamma_0}_{L^2(dxdy)}.
\end{align}
\end{prop}

\begin{proof}
The proof is essentially the same as the standard non-endpoint Strichartz estimate using both the $TT^\ast$ principle and Christ-Kiselev lemma. See \textsection 2.3 in \cite{Tao}.  
\end{proof}

\section{Estimates for the Inhomogeneous $\Gamma$ Equation}\label{sec-4}
Let us now consider the inhomogeneous $\Gamma$ equation 
\begin{align}\label{gamma-f}
\vect{S}_\pm \Gamma = F
\end{align} 
where $F$ is smooth. The main purpose of this section is to obtain collapsing estimates similar to estimates proven in
Proposition \ref{collapsing-est} and Corollary \ref{homogeneous-rho-est} but for that inhomogeneous equation. The main results of this section are Proposition \ref{Inhomog-eq} and Proposition \ref{inhomog-eq2}.
\begin{remark}
For the purpose of obtaining estimates for (\ref{gamma-f}), we do not need to assume $F$ to have any symmetry. That being said, in order for our iteration scheme to preserve the symmetry $\Gamma(x, y) = \overline{\Gamma(y, x)}$, i.e. stay in the designated space which we have specified in Theorem \ref{main}, it is wise to assume $F$ is skewed symmetric, i.e. $\overline{F(x, y)} =-F(y, x)$. Likewise, the forcing term with respect to the $\Lambda$ equation should also satisfy $F(x, y) = F(y, x)$.  Henceforth, we assume the forcing $F$ for each of the three equations has the correct symmetry. 
\end{remark}

Observe the solution to the inhomogeneous equation can be written as
\begin{align}
\Gamma(t, x, y) = e^{it(\lapl_x-\lapl_y)}\Gamma_0(x, y)+ i\int^t_0 e^{i(t-s)(\lapl_x-\lapl_y)}F(s, x, y)\ ds
\end{align}
which then yields
\begin{align}\label{rho-eq}
\rho_\Gamma(t, x) = [e^{it(\lapl_x-\lapl_y)}\Gamma_0](x, x)+ i\int^t_0 [e^{i(t-s)(\lapl_x-\lapl_y)}F](s, x, x)\ ds
\end{align}
Then it follows from the estimate (\ref{col-est}) that
\begin{align*}
 \llp{\grad^{1/2}_x\rho_\Gamma}_{L^2(dtdx)}
\lesssim&\ \llp{\Gamma_0}_{L^2(dxdy)}+ \int^T_0 
\llp{\grad^{1/2}_x[e^{i(t-s)(\lapl_x-\lapl_y)}F](s, x, x)}_{L^2(dtdx)}\ ds \\
\lesssim&\  \llp{\Gamma_0}_{L^2(dxdy)}+ \llp{F}_{L^1[0,T]L^2(dxdy)}.
\end{align*}
Hence we have obtained the following proposition
\begin{prop}\label{Inhomog-eq}
Suppose $\Gamma$ solves $\vect{S}_\pm \Gamma = F$, then we have 
\begin{align}
 \llp{\grad^{1/2}_x\rho_\Gamma}_{L^2(dtdx)}
 \lesssim \llp{\Gamma_0}_{L^2(dxdy)}+ \llp{F}_{L^1[0, T]L^2(dxdy)}.
\end{align}
\end{prop}
The following is a perturbed version of the above proposition. 
\begin{prop}\label{inhomog-eq2}
 Suppose $\Gamma$ solves $\vect{S}_\pm \Gamma = F$, then for every $0<\varepsilon<\frac{1}{2}$ we have 
\begin{align}
 \llp{\grad^{\frac{1}{2}-\varepsilon}_x\rho_\Gamma}_{L^2(dtdx)}\lesssim&\ T^\text{some power}
\bigg(\llp{\grad^{\sigma}_{x, y} \Gamma(t, x, y)}_{L^\infty[0, T]L^2(dxdy)} \nonumber\\
&\ +\llp{\grad^{\sigma}_{x,y}F}_{L^1[0, T]L^2(dxdy)}\bigg).
\end{align}
\end{prop}

\begin{proof}
Applying Corollary \ref{homogeneous-rho-est} to (\ref{rho-eq}) yields
\begin{align*}
 \llp{\grad^{\frac{1}{2}-\varepsilon}_x&\rho_\Gamma(t, x)}_{L^2(dtdx)} \\
\lesssim&\ T^\text{some power}
\bigg(\llp{\grad^{\sigma}_{x,y}\Gamma(t,x, y)}_{L^\infty[0, T]L^2(dxdy)}\\
&+ \int^T_0 ds\
\llp{\grad^{\frac{1}{2}-\varepsilon}_x[e^{i(t-s)(\lapl_x-\lapl_y)}F](s, x, x)}_{L^q[0, T]L^2(dx)}\bigg) \\
\lesssim&\  T^\text{some power}
\left(\llp{\grad^{\sigma}_{x,y}\Gamma(t,x, y)}_{L^\infty[0, T] L^2(dxdy)}
+ \llp{\grad_{x,y}^{\sigma}F}_{L^1[0,T]L^2(dxdy)}\right)
\end{align*}
where $q= \frac{2}{(2-\varepsilon)(\frac{1}{2}-\varepsilon)}$. Then by Remark \ref{remark-1} we obtain the desired estimate.
\end{proof}

To conclude this section, let us state the inhomogeneous Strichartz estimate.
\begin{prop}\label{gamma-stric}
Suppose $\Gamma$ is a solution to $\vect{S}_\pm\Gamma = F$ with initial condition $\Gamma_0$ and $(k, \ell)$ and 
$(\tilde k, \tilde \ell)$ are an admissible pairs (see (\ref{ad-pair})). Then it follows
\begin{align}
\llp{\Gamma(t, x, y)}_{L^k(dt)L^\ell(dx)L^2(dy)}\lesssim \llp{\Gamma_0}_{L^2(dxdy)}+\llp{F}_{L^{\tilde k'}(dt)L^{\widetilde \ell'}(dx)L^2(dy)}
\end{align}
and
\begin{align}
 \llp{\grad_{x, y}^{\sigma}\Gamma(t, x, y)}_{L^k(dt)L^\ell(dx)L^2(dy)}\lesssim 
\llp{\grad_{x, y}^{\sigma}\Gamma_0}_{L^2(dxdy)}
+\llp{\grad_{x, y}^{\sigma}F}_{L^{\widetilde k'}(dt)L^{\widetilde \ell'}(dx)L^2(dy)}
\end{align}
where $(\widetilde k', \widetilde \ell')$ denotes the H\" older conjugates of $(\widetilde k, \widetilde \ell)$. 
\end{prop}

\section{Application of the Inhomogeneous $\Gamma$ Estimates }
The purpose of this section is to develop estimates which we will later use in the proof of our main theorem in \textsection \ref{Proof-of-main}. 
However, as an immediate application of the previous two sections, we are now
ready to consider the uniform in $N$ local well-posedness of the following Hartree-Fock equation 
\begin{align}\label{Hartree}
\frac{1}{i}\frac{\bd}{\bd t}\Gamma = [\lapl - v_N\ast\rho_\Gamma, \Gamma]
\end{align}
or equivalently
\begin{align}
\vect{S}_\pm \Gamma(t, x, y) = [v_N\ast\rho_\Gamma(t, x)-v_N\ast\rho_\Gamma(t, y)]\Gamma(t, x, y) = F
\end{align}
in some Strichartz-type space $X$ equipped with the norm
\begin{align}\label{norm}
\llp{\Gamma}_X:=&\  \llp{\grad^{\frac{1}{2}-\varepsilon}_x\rho_\Gamma(t, x)}_{L^2[0, T] L^2(dx)}+\llp{\grad^{1/2}_x\rho_\Gamma(t, x)}_{L^2[0, T] L^2(dx)}\\
&\ +\llp{\langle\grad_{x,y}\rangle^{\sigma}\Gamma(t, x, y)}_{L^\infty[0, T] L^2(dxdy)}
+\llp{\langle\grad_{x,y}\rangle^{\sigma}\Gamma(t, x, y)}_{L^4[0, T] L^\infty(dx)L^2(dy)} \nonumber
\end{align}
where $\varepsilon$ is sufficiently small, say $\varepsilon<\frac{1}{5}$. 

The uniform in $N$ local well-posedness is proven using the standard Banach fixed-point argument. More precisely, we close the estimate for (\ref{Hartree}) in $X$. In particular, we close the estimate for each of the three norms indicated in (\ref{norm}).  However, by Proposition \ref{Inhomog-eq} and Proposition \ref{inhomog-eq2}, it suffices to consider estimates for the corresponding forcing terms. 

 First, let us estimate $\llp{F}_{L^1[0, T]L^2(dxdy)}$. By H\" older's inequality, we see that
\begin{align*} 
 \llp{F}_{L^1[0, T]L^2(dxdy)} \lesssim&\ \llp{v_N\ast\rho_\Gamma(t, x)
\Gamma(t, x, y)}_{L^1[0, T]L^2(dxdy)}\\
 \lesssim&\ \llp{v_N\ast\rho_\Gamma(t, x)}_{L^2(dt)L^{p}(dx)}\llp{\Gamma(t, x, y)}_{L^{2}(dt)L^{r}(dx)L^2(dy)}
 \end{align*}
 where we made the choice $p = 1/\varepsilon$ and $r = 2(1-2\varepsilon)^{-1}$. Then by Gagligardo-Nirenberg-Sobolev inequality, Young's convolution inequality and H\" older inequality, in the time variable, we obtain the estimate
 \begin{align*}
\llp{F}_{L^1[0, T]L^2(dxdy)} \lesssim&\ \llp{v_N\ast\grad_x^{\frac{1}{2}-\varepsilon}\rho_\Gamma(t, x)}_{L^2(dtdx)}\llp{\Gamma(t, x, y)}_{L^{2}(dt)L^{r}(dx)L^2(dy)}\\
\lesssim&\ T^\text{some power}\llp{\grad_x^{\frac{1}{2}-\varepsilon}\rho_\Gamma(t, x)}_{L^2(dtdx)}\llp{\Gamma(t, x, y)}_{L^{q}(dt)L^{r}(dx)L^2(dy)}
\end{align*}
where $q = 2\varepsilon^{-1}$. Note that $q$ is chosen  so that $(q, r)$ is a 1D Strichartz admissible pair. Hence by interpolation, we see that
\begin{align*}
\llp{F}_{L^1[0, T]L^2(dxdy)} \lesssim T^\text{some power}\llp{\Gamma}_X^2.
\end{align*}
Likewise, we can show that $\llp{\grad^{1/2}_{x, y} F}_{L^1[0, T]L^2(dxdy)}$ also closes. 

Next, let us estimate $ \llp{\grad_{x, y}^{\sigma}F}_{L^1[0, T]L^2(dxdy)}$. Using the classical Kato-Ponce inequality, sometimes refers to as ``fractional Leibniz rule", we see that  
\begin{align*}
 \llp{\grad_{x, y}^{\sigma}F}_{L^1[0, T]L^2(dxdy)} \lesssim&\
 \llp{\grad^{\sigma}_{x, y}[v_N\ast\rho_\Gamma(t, x)
\Gamma(t, x, y)]}_{L^1[0, T]L^2(dxdy)}\\
\lesssim&\ \llp{ v_N\ast \rho_\Gamma(t, x)}_{L^2(dt)L^p(dx)}
\llp{\grad^{\sigma}_{y}\Gamma(t, x, y)}_{L^2(dt)L^r(dx)L^2(dy)}\\ 
&+\llp{  v_N\ast \grad^{\sigma}_{x}\rho_\Gamma(t, x)}_{L^2(dt)L^{\widetilde p}(dx)}
\llp{\Gamma(t, x, y)}_{L^2(dt)L^{\widetilde r}(dx)L^2(dy)}\\
&+\ \llp{v_N\ast\rho_\Gamma(t, x) }_{L^2(dt)L^{p}(dx)}\llp{\grad^{\sigma}_{x} 
\Gamma(t, x, y)}_{L^2[0, T]L^r(dx)L^2(dy)}
\end{align*}
where $\widetilde p=2[(5-2\varepsilon)\varepsilon]^{-1}, \widetilde r = [\varepsilon^2-\frac{5}{2}\varepsilon+\frac{1}{2}]^{-1}$ and $p, r$ as defined above. Hence by the same argument as above with $\widetilde q = 2[(\frac{5}{2}-\varepsilon)\varepsilon]^{-1}$ we see that
\begin{align*}
 \llp{\grad^{\sigma}_{x, y}&[v_N\ast\rho_\Gamma(t, x)\Gamma(t, x, y)]}_{L^1[0, T]L^2(dxdy)}\\
 \lesssim&\ T^\text{some power}\llp{\grad^{\frac{1}{2}-\varepsilon}_x\rho_\Gamma(t, x)}_{L^2(dtdx)} \llp{\grad^{\sigma}_{y}\Gamma(t, x, y)}_{L^q(dt)L^r(dx)L^2(dy)}\\ 
 &+ T^\text{some power}\llp{\grad^{\frac{1}{2}-\varepsilon}_x \rho_\Gamma(t, x)}_{L^2(dtdx)} \llp{\Gamma(t, x, y)}_{L^{\widetilde q}(dt)L^{\widetilde r}(dx)L^2(dy)}\\
 &+ T^\text{some power}\llp{\grad^{\frac{1}{2}-\varepsilon}_x \rho_\Gamma(t, x)}_{L^2(dtdx)} \llp{\grad^{\sigma}_{x}\Gamma(t, x, y)}_{L^q(dt)L^r(dx)L^2(dy)}
\end{align*}
Again, note that $(\widetilde q, \widetilde r)$ is an admissible pair which means the desired estimate holds by interpolation.
\begin{remark}
The below estimate is included in this section purely for the author's own organizational purposes. Hence the reader may skip it for now and refer back to it in \textsection 8.
\end{remark}

Lastly, observe we have
\begin{align*}
 \llp{\grad_{x+y}^{\sigma} F}_{L^1[0, T]L^2(dxdy)} \lesssim&\ 
\llp{\grad_{x+y}^{\sigma}
[v_N\ast \rho_\Gamma(t, x)\Gamma(t, x, y)]}_{L^1[0, T]L^2(dxdy)}\\
\lesssim&\ \llp{  v_N\ast\grad^{\sigma}_{x}\rho_\Gamma(t, x)}_{L^2(dt)L^{\widetilde p}(dx)}
\llp{\Gamma(t, x, y)}_{L^2(dt)L^{\widetilde r}(dx)L^2(dy)}\\
&+\ \llp{v_N\ast\rho_\Gamma(t, x) }_{L^2(dt)L^{p}(dx)}\llp{\grad^{\sigma}_{x+y} 
\Gamma(t, x, y)}_{L^2[0, T]L^r(dx)L^2(dy)}\\
\lesssim&\ T^\text{some power}\llp{\grad^{\frac{1}{2}-\varepsilon}_x \rho_\Gamma(t, x)}_{L^2(dtdx)} \llp{\Gamma(t, x, y)}_{L^{\widetilde q}[0, T]L^{\widetilde r}(dx)L^2(dy)}\\
 &+ T^\text{some power}\llp{\grad^{\frac{1}{2}-\varepsilon}_x \rho_\Gamma(t, x)}_{L^2(dtdx)} \llp{\grad^{\sigma}_{x}\Gamma(t, x, y)}_{L^q[0, T]L^r(dx)L^2(dy)}
\end{align*}
where $\grad_{x+y} F:= \frac{1}{2}(\grad_x F+\grad_y F)$. 

\begin{remark}\label{pqr}
Since similar calculations will be performed in \textsection 8, then for convenience  we shall fix the values of $p, q, r, \widetilde p, \widetilde q, \widetilde r$ as indicated above for a given $\varepsilon$ in the remaining of this article.
\end{remark}

As a result of the above calculation, we obtain the following proposition
\begin{prop}
 Suppose $\Gamma$ solves (\ref{Hartree}) with Schwartz initial condition $\Gamma_0$ and $v \in L^1(\rr)$.  Then the following estimate holds
 \begin{align*}
 \llp{\Gamma}_X \lesssim \llp{\langle\grad_{x,y}\rangle^{\sigma}\Gamma_0}_{L^2(dxdy)}+T^{\text{some power}}\llp{\Gamma}_X^2.
 \end{align*}
Thus, there exists $T_0>0$ such that for all $0<T\leq T_0$
\begin{align*}
\llp{\Gamma}_X \lesssim \llp{\langle\grad_{x,y}\rangle^{\sigma}\Gamma_0}_{L^2(dxdy)}.
\end{align*}
Similarly, we can show that
\begin{align*}
 \llp{\bd_t\Gamma}_X \lesssim \llp{\langle\grad_{x,y}\rangle^{\sigma}\bd_t\Gamma_0}_{L^2(dxdy)}+T^{\text{some power}}\llp{\Gamma}_X\llp{\bd_t\Gamma}_X
\end{align*}
which again means there exists $T_0>0$ such that
\begin{align*}
\llp{\bd_t\Gamma}_X \lesssim \llp{\langle\grad_{x,y}\rangle^{\sigma}\bd_t\Gamma_0}_{L^2(dxdy)}.
\end{align*}
\end{prop}

 \section{Homogeneous $\Lambda$ Equation}
In this section we prove collapsing estimates for the linear Schr\" odinger equations
\begin{align}
\frac{1}{i}\frac{\bd}{\bd t}\Lambda -\lapl_x \Lambda -\lapl_y\Lambda = 0
\end{align}
which we will need later. As mentioned in the introduction, one of the main difficulties 
in the analysis of equation (\ref{nl-lambda}) is that the $L^p$-norms of the potential $N^{-1}v_N(x-y)$
are  not uniformly bounded in $N$ when $p>1$ and $\beta$ arbitrarily large
 since $N^{-1}\llp{v_N(x-y)}_p \sim N^{-1+\beta(1-\frac{1}{p})}$.  More precisely, 
from Proposition \ref{weird-deriv}, we see that the natural space to put the 
nonlinearity of equation (\ref{nl-Lam}) 
is in $L^1[0, T]L^2(dxdy)$. In particular, when handling the term 
$N^{-1}v_N(x-y) \Lambda(t, x, y)$ from equation (\ref{nl-lambda}) in $L^1([0, T]\times L^2(\rr^2))$, 
we see there is no way (at least no simple way)
 to put the term $N^{-1}v_N(x-y)$ in $L^1(d(x-y))$. Thus, the 
purposes of \textsection $6$ and \textsection $7$ are to develop sufficient amount of tools to
 handle $N^{-1}v_N(x-y)\Lambda(t, x, y)$ and all 
the nonlinearity coming from the TDHBF equations.

One of the crucial tools for our analysis is the $X^{s, b}$ spaces (sometimes called the 
\emph{Bourgain spaces} or \emph{dispersive Sobolev spaces}) which is 
defined to be the closure of the Schwartz class, $\mathcal{S}_{t, x}(\rr\times\rr\times\rr)$ with
respect to the norm
\begin{align*}
 \llp{u}_{X^{s, b}_{\vect{S}}}= \llp{(1+ |\xi|^2+|\eta|^2)^{s}
(1+|\tau+|\xi|^2+|\eta|^2|)^{b} \tilde u(\tau,\xi, \eta) }_{L^2(d\tau)L^2(d\xi d\eta)}.
\end{align*}
For this paper, $s$ is always zero and we are only interested in 
defining the $X^{s, b}$ spaces for the operator $\vect{S}$. 
Hence we dropped both the $s$ and $\vect{S}$ labels from the norm to simplify the notation.
 For instance, we have $\llp{u}_{X^b} = \llp{u}_{X^{0, b}_\vect{S}}$.  We
refer the interested reader to \textsection 2.6 in \cite{Tao} for a more complete introduction to these spaces.     

 Same as the von-Neumann Schr\" odinger equation, we first obtain a collapsing estimate for the above equation.  
\begin{prop}\label{Lambda-collapse}
Suppose $\vect{S}\Lambda = 0$ with Schwartz initial condition $\Lambda(0, x, y) = \Lambda_0(x, y)$ then
\begin{align}\label{L-collapse}
 \llp{p\left(\bd_t, \grad_x\right)\Lambda(t, x, x)}_{L^2(dtdx)} \lesssim \llp{\Lambda_0}_{L^2(dxdy)}.
\end{align}
where $p(\bd_t, \grad_x)$ is a pseudodifferential operator with symbol $\tilde p(\tau, \xi) = |\tau+|\xi|^2|^{1/4}$.  
\end{prop}

\begin{proof}
 Let us begin by taking the spacetime Fourier transform of the trace of $\Lambda$ to get
\begin{align*}
 \widetilde{\Lambda(t, x, x)} =& \int dtdx\ e^{-i(\tau t+\xi\cdot x)}\Lambda(t, x, x) 
=\ \int dtdxdy\ e^{-i(\tau t+\xi\cdot x)}\delta(x-y)\Lambda(t, x, y)\\
=&\ \frac{1}{2\pi}\int d\eta dtdxdy\ e^{-i(\tau t+(\xi-\eta)\cdot x+\eta\cdot y)}\Lambda(t, x, y) 
=\ \frac{1}{2\pi}\int d\eta dt\ e^{-i\tau t}\widehat{\Lambda}(t, \xi-\eta, \eta)\\
=&\ \frac{1}{2\pi}\int d\eta\ \delta(\tau+|\xi-\eta|^2+|\eta|^2)\widehat{\Lambda_0}(\xi-\eta, \eta).
\end{align*}
Applying Cauchy-Schwarz inequality yields the following estimate
\begin{align*}
 \int d\tau d\xi\ |(\tau+|\xi|^2)^{1/4}\widetilde{\Lambda(t, x, x)}(\tau, \xi)|^2 
\lesssim\ \sup_{\tau, \xi}|I(\tau, \xi)|\llp{\Lambda_0}_{L^2(dxdy)}^2
\end{align*}
where
\begin{align*}
I(\tau, \xi) := \sqrt{\tau+|\xi|^2}\int d\eta\ \delta(\tau +|\xi-\eta|^2+|\xi+\eta|^2).
\end{align*}
Observe, we have the identity
\begin{align*}
\int d\eta\ \delta(\tau +|\xi-\eta|^2+|\xi+\eta|^2) &= \int_{\rr} 
\frac{\delta(\eta-\sqrt{-\tau-|\xi|^2})+\delta(\eta+\sqrt{-\tau-|\xi|^2})}{4\sqrt{-\tau-|\xi|^2}}\ d\eta \\
&= \frac{1}{2\sqrt{-\tau -|\xi|^2}}.
\end{align*}
Thus, it follows
\begin{align*}
  \int d\tau d\xi\ ||\tau+|\xi|^2|^{1/4}\widetilde{\Lambda(t, x, x)}(\tau, \xi)|^2 
\lesssim \llp{\Lambda_0}_{L^2(dxdy)}^2.
\end{align*}
\end{proof}

Unfortunately, the homogeneous derivative $p(\bd_t, \grad_x)$ of the restriction of $\Lambda$ to the diagonal is not of any immediate use to our studies of the nonlinear coupled equations. Since the nonlinearity in TDHFB involves trace of $\Lambda$, we need estimates that will allow us to control the restricted $\Lambda$ 
by the spacetime derivative $p(\bd_t, \grad_x)$ of the restriction of $\Lambda(t, x, y)$ to the diagonal. One such estimate is given by the following proposition. 
\begin{prop}\label{weird-deriv}
 Suppose $\vect{S}\Lambda =0$, then we have
\begin{align}
 \llp{\Lambda(t, x, x)}_{L^4(dt)L^2(dx)} \lesssim \llp{p(\bd_t, \grad_x)\Lambda(t, x, x)}_{L^2(dtdx)}
\end{align}
\end{prop}

\begin{proof}
We prove the above estimate using a $TT^\ast$ argument. Consider $T: L^2_{t, x} \rightarrow L^4_tL^2_x$ defined by
\begin{align*}
 (T F)(t, x) = \left(\frac{\widetilde F}{|\tau+|\xi|^2|^{1/4}}\right)^\vee := \frac{1}{2\pi} \int d\xi\ e^{i(\xi\cdot x+\tau t)} \frac{\widetilde{F}(\tau, \xi)}{|\tau+|\xi|^2|^{1/4}}
\end{align*}
then we see that $TT^\ast: L^{4/3}_tL^2_x \rightarrow L^4_tL^2_x$ is given by 
\begin{align*}
 TT^\ast F = \left(\frac{\widetilde F}{|\tau+|\xi|^2|^{1/2}}\right)^\vee 
= F\ast\left(\frac{1}{|\tau+|\xi|^2|^{1/2}} \right)^\vee =: F\ast K.
\end{align*}
By triangle inequality and Plancherel, we obtain the estimate
\begin{align*}
 \Lp{K\ast F(t, \cdot)}{L^2(dx)} 
\leq&\ \int ds\ \llp{\hat K(t-s,\xi) \hat F(s, \xi)}_{L^2(d\xi)} 
\lesssim\  \int ds\ \frac{1}{|t-s|^{1/2}}\llp{\hat F(s,\cdot)}_{L^2(d\xi)}
\end{align*}
since we have
\begin{align*}
|\hat K(t-s,\xi)|
=&\ \left|\int^\infty_{-\infty} e^{i\tau(t-s)} \frac{e^{i|\xi|^2(t-s)}}{|\tau|^{1/2}}\ d\tau \right|\lesssim \frac{1}{|t-s|^{1/2}}
\end{align*}
which is independent of $\xi$.
Thus, it follows 
\begin{align*}
 \llp{TT^\ast F}_{L^4(dt)L^2(dx)} \lesssim \Lp{\int^\infty_{-\infty} ds\ \frac{\llp{\hat F(s, \cdot)}_{L^2(d\xi)}}{|t-s|^{1/2}}}{L^4(dt)}.
\end{align*}
Now, apply Hardy-Littlewood-Sobolev inequality  $\frac{n}{p}=\frac{n}{q}-n+\alpha$, with $n=1, p=4/3$ and $q=4$
we have that 
\begin{align*}
 \Lp{\int^\infty_{-\infty} ds\ \frac{\llp{\hat F(s, \cdot)}_{L^2(d\xi)}}{|t-s|^{1/2}}}{L^4(dt)} \lesssim \llp{F}_{L^{4/3}(dt)L^2(dx)}
\end{align*}
which means $TT^\ast$ is a bounded operator. Hence it follows from the $TT^\ast$ principle that $T$ is also a bounded operator, i.e.
\begin{align*}
 \llp{TF}_{L^4(dt)L^2(dx)} \lesssim \llp{F}_{L^2(dtdx)}
\end{align*}
or equivalently
\begin{align*}
 \llp{F}_{L^4(dt)L^2(dx)} \lesssim \llp{|\tau+|\xi|^2|^{1/4}\tilde F(\tau, \xi)}_{L^2(d\tau d\xi)}.
\end{align*}
\end{proof}

As an immediate corollary of Proposition \ref{weird-deriv}, we have that
\begin{cor}
 Suppose $\Lambda$ solves $\vect{S}\Lambda = 0$, then for every $0<\varepsilon<1$ we have
\begin{align}
 \llp{\grad_x^\varepsilon\Lambda(t, x, x)}_{L^4(dt)L^2(dx)}
\lesssim \llp{\grad_{x+y}^\varepsilon\Lambda_0}_{L^2(dxdy)}
\end{align}
where $\grad_{x+y}\Lambda:=\frac{1}{2}(\grad_x\Lambda +\grad_y\Lambda) $. 
\end{cor}
\begin{proof}
 If $\vect{S}\Lambda = 0$, then $\vect{S}\grad_{x+y}\Lambda = 0$. Applying the previous estimate, we obtain the estimate
\begin{align*}
 \llp{(\grad_{x+y}\Lambda)(t, x, x)}_{L^4(dt)L^2(dx)} \lesssim&\ \llp{p(\bd_t, \grad_x)
(\grad_{x+y}\Lambda)(t, x, x)}_{L^2(dtdx)}\\
\lesssim&\ \llp{\grad_{x+y}\Lambda_0}_{L^2(dxdy)}.
\end{align*}
Noting the identity
\begin{align}\label{deriv-id}
 (\grad_{x+y}\Lambda)(t, x, x) = \frac{1}{2}\grad_x\left(\Lambda(t, x, x) \right),
\end{align}
we get the estimate
\begin{align*}
  \llp{\grad_{x}\Lambda(t, x, x)}_{L^4(dt)L^2(dx)} \lesssim \llp{
\grad_{x+y}\Lambda_0}_{L^2(dxdy)}.
\end{align*}
Interpolating the above estimate with the estimate
\begin{align*}
   \llp{\Lambda(t, x, x)}_{L^4(dt)L^2(dx)} \lesssim \llp{\Lambda_0}_{L^2(dxdy)}
\end{align*}
yields the desired result.
\end{proof}

Let us also record the following non-endpoint Strichartz estimate for the homogeneous $\Lambda$ equation
\begin{prop}[Non-endpoint Strichartz]\label{strichartz2}
Suppose $\Lambda$ is a solution to $\vect{S}\Lambda = 0$ with initial condition 
$\Lambda_0$ and $(k, \ell)$ is an admissible pair as defined in Proposition \ref{strichartz1}. Then it follows
\begin{align}
\llp{e^{it\lapl_x}\Lambda_0e^{it\lapl_y}}_{L^k(dt)L^\ell(dx)L^2(dy)}\lesssim \llp{\Lambda_0}_{L^2(dxdy)}.
\end{align}
\end{prop}

\begin{prop}\label{X-dual-est}
For any number $1^+>1$ and arbitrarily close to $1$ there exists $\delta>0$ such that the following estimate holds
\begin{align}\label{dual-X-est}
 \llp{F}_{X^{-1/2+\delta}} \lesssim T^\text{some power}\llp{F}_{L^2[0, T]L^{1^+}(dx)L^2(dy)}.
\end{align}
\end{prop}

\begin{proof}
 By Proposition \ref{strichartz2} and Lemma 2.9 in \cite{Tao}, we have the estimate
\begin{align}\label{dispersive1}
 \llp{F}_{L^4[0,T]L^\infty(dx)L^2(dy)} \lesssim \llp{F}_{X^{1/2+\delta}}
\end{align}
for all $\delta>0$. Moreover, from (\ref{dispersive1}) we also get the dual estimate
\begin{align}\label{dual-dispersive1}
 \llp{F}_{X^{-1/2-\delta}} \lesssim \llp{F}_{L^{4/3}[0, T]L^1(dx)L^2(dy)} 
\lesssim T^{1/4}\llp{F}_{L^2[0, T]L^1(dx)L^2(dy)}.
\end{align}
By linearly interpolating (\ref{dual-dispersive1}) with
\begin{align*}
 \llp{F}_{X^{-1/2+1/2}} = \llp{F}_{L^2[0, T]L^2(dx)L^2(dy)}
\end{align*}
yields
\begin{align*}
 \llp{F}_{X^{-1/2+\lambda}} \lesssim T^\text{some power}\llp{F}_{L^2[0, T]L^{1^+}(dx)L^2(dy)}
\end{align*}
for $-\delta<\lambda<\frac{1}{2}$ and some number $1^+$ depending on $\lambda$. 
In particular, for any number $1^+$ arbitrarily close to $1$
we can choose $\delta$ sufficiently small such that (\ref{dual-X-est}) holds.  
\end{proof}

\begin{remark}
Let us make the observation: since 
\begin{align}
|\xi +\eta|^2+|\xi-\eta|^2 = 2|\xi|^2+2|\eta|^2
\end{align}
then we also have the estimate
\begin{align*}
 \llp{F}_{X^{-1/2+\delta}} \lesssim T^\text{some power}\llp{F}_{L^2[0, T]L^{1^+}(d(x-y))L^2(d(x+y))}.
\end{align*}

\end{remark}

\section{Inhomogeneous $\Lambda$ Equation}
The main result in this section is Corollary \ref{main-cor} which allows us to obtain a collapsing-type estimate for equation (\ref{nl-lambda}) and essentially show that $N^{-1}v_N\Lambda$, mentioned in the previous section, can be viewed as a uniformly in $N$ perturbation of equation (\ref{nl-Lam}). 

Consider the inhomogeneous equation
\begin{align}\label{nl-Lam}
\vect{S}\Lambda = F
\end{align}
then it follows from the $X^{s, b}$ energy estimate\footnote{cf. \cite{Tao} section 2.6} and
Proposition \ref{X-dual-est} that we have
\begin{align*}
 \llp{\Lambda(t, x, x)}_{L^4(dt)L^2(dx)}
\lesssim&\ \llp{\Lambda_0}_{L^2(dxdy)} + \llp{F}_{X^{-1/2+\delta}}\\
\lesssim&\ \llp{\Lambda_0}_{L^2(dxdy)} + T^\text{some power}\llp{F}_{L^{2}(dt)L^{1^+}(d(x-y))L^2(d(x+y))}
\end{align*}
Summarizing the above result we obtain the following proposition
\begin{prop}\label{Inhomog-eq2}
Suppose $\Lambda$ solves $\vect{S} \Lambda = F$, then we have 
\begin{align}
  \llp{\Lambda(t, x, x)}_{L^4(dt)L^2(dx)} \lesssim  \llp{\Lambda_0}_{L^2(dxdy)} 
+ T^\text{some power}\llp{F}_{L^{2}(dt)L^{1^+}(d(x-y))L^2(d(x+y))}.
\end{align}
\end{prop}

Using the above proposition, we establish the following proposition
\begin{prop}\label{trace-energy-lam}
 Suppose $\Lambda$ solves (\ref{nl-lambda})
with initial condition $\Lambda_0$. Then we have
\begin{align}\label{trace-energy-est}
\llp{\Lambda(t, x, x)}_{L^4(dt)L^2(dx)} \lesssim \llp{\Lambda_0}_{L^2(dxdy)}+\llp{F}_{X^{-1/2+\delta}}.
\end{align}
\end{prop}

\begin{proof}
Since by Proposition \ref{weird-deriv} we have
\begin{align*}
\Lp{e^{it\lapl_x}\Lambda_0 e^{it\lapl_y}\big|_{x=y}}{L^2(dtdx)} \lesssim \llp{\Lambda_0}_{L^2(dxdy)},
\end{align*}
then it follows from Lemma 2.9 in \cite{Tao}. 
\begin{align*}
\llp{F}_{L^4(dt)L^2(dx)}\lesssim \llp{F}_{X^{1/2+\delta}}
\end{align*}
for any $\delta>0$.
In particular, applying the $X^{s,b}$ energy estimate we get that
\begin{align*}
\llp{\bigchi(t) \Lambda}_{X^{1/2+\delta}} \lesssim&\ \llp{\frac{1}{N}v_N\bigchi(t)\Lambda}_{X^{-1/2+\delta}}
\ + \llp{F}_{X^{-1/2+\delta}}+\llp{\Lambda_0}_{L^2(dxdy)}
\end{align*}
where $\chi(t)$ is a time localization bump function. 
 Applying Proposition \ref{X-dual-est}, we see that
\begin{align*}
 \frac{1}{N}\llp{v_N(x-y)\Lambda}_{X^{-1/2+\delta}} \lesssim&\ \frac{1}{N}T^\text{some power}\llp{v_N\bigchi(t)\Lambda}_{L^{2}(dt)L^{1^+}(d(x-y))L^2(d(x+y))}\\
\lesssim&\ \frac{1}{N}T^\text{some power}\llp{v_N}_{L^{1^+}(d(x-y))}\llp{\bigchi(t)\Lambda}_{L^{2}(dt)L^\infty(d(x-y))L^2(d(x+y))}\\
\lesssim&\ \frac{1}{N^{1-\beta+\beta/(1^+)}} T^{\text{some power}}\llp{\bigchi(t)\Lambda}_{L^4(dt)L^\infty(d(x-y))L^2(d(x+y))}\\
\lesssim&\ \frac{1}{N^{1-\beta+\beta/(1^+)}} T^{\text{some power}}\llp{\bigchi(t)\Lambda}_{X^{1/2+\delta}}
\end{align*}
Hence for $1^+$ sufficiently close to $1$ we are in the perturbative regime.
 This allows us to absorb the contribution from the potential 
term $\frac{1}{N}v_N(x-y)\Lambda$ when $N$ is sufficiently large. 
\end{proof}

Using the above proposition we could show that
\begin{cor}\label{main-cor}
 Suppose $\Lambda$ solves (\ref{nl-lambda})
with initial condition $\Lambda_0$. Then for every $0<\sigma<1/2$ we have
\begin{align}
\llp{\grad_{x}^\sigma\Lambda(t, x, x)}_{L^4(dt)L^2(dx)} \lesssim \llp{\grad_{x+y}^\sigma\Lambda_0}_{L^2(dxdy)}+\llp{\grad_{x+y}^\sigma F}_{X^{-1/2+\delta}}.
\end{align}
\end{cor}
\begin{proof}
Taking the spatial derivative $\grad_{x+y}$ of (\ref{nl-lambda}) yields
\begin{align}\label{xplusy-deriv-lamb}
\left(\vect{S}+\frac{1}{N}v_N(x-y)\right)\grad_{x+y}\Lambda= \grad_{x+y}F
\end{align}
since $[\grad_{x+y}, N^{-1}v_N(x-y)]=0$. Hence by Proposition \ref{trace-energy-lam}, we obtain the estimate
\begin{align*}
\llp{(\grad_{x+y}\Lambda)(t, x, x)}_{L^4(dt)L^2(dx)} \lesssim \llp{\grad_{x+y}\Lambda_0}_{L^2(dxdy)}+\llp{\grad_{x+y}F}_{X^{-1/2+\delta}}.
\end{align*}
Again, noting the identity (\ref{deriv-id}), we obtain the estimate
\begin{align}\label{deriv-trace-energy-est}
\llp{\grad_{x}\Lambda(t, x, x)}_{L^4(dt)L^2(dx)} \lesssim \llp{\grad_{x+y}\Lambda_0}_{L^2(dxdy)}+\llp{\grad_{x+y}F}_{X^{-1/2+\delta}}.
\end{align}
Interpolating (\ref{trace-energy-est}) with (\ref{deriv-trace-energy-est}) yields the desired result. 
\end{proof}

Now, let us record some Strichartz estimates
\begin{prop}\label{strichartz3}
 Suppose $\Lambda$ is a solution to $\vect{S}\Lambda = F$ with initial condition 
$\Lambda_0$ and $(k, \ell), (\widetilde k, \widetilde \ell)$ are Strichartz admissible pairs. Then it follows
\begin{align}
\llp{\Lambda(t, x, y)}_{L^k(dt)L^\ell(dx)L^2(dy)}\lesssim \llp{\Lambda_0}_{L^2(dxdy)}+\llp{F}_{L^{\widetilde k'}(dt)L^{\widetilde \ell'}(dx)L^2(dy)}.
\end{align}
In particular, it follows
\begin{align}
 \llp{\grad^\sigma_{x, y}\Lambda(t, x, y)}_{L^k(dt)L^\ell(dx)L^2(dy)}\lesssim \llp{\grad_{x, y}^\sigma\Lambda_0}_{L^2(dxdy)}+\llp{\grad_{x, y}^\sigma F}_{L^{\widetilde k'}(dt)L^{\widetilde \ell'}(dx)L^2(dy)}.
\end{align}
\end{prop}

\begin{remark}\label{beta}
  Let us note that Proposition \ref{strichartz3} also holds for solution to (\ref{nl-lambda}) when $N$ is sufficiently large. More specifically, by interpolation, we can show
\begin{align}
\frac{1}{N}\llp{\grad_x^\sigma[v_N(x-y)]\Lambda}_{L^{4/3}[0, T]L^1(d(x-y))L^2(d(x+y))}\lesssim \frac{T^{1/2}}{N^{1-\sigma\beta}} \llp{\Lambda}_{L^{4}[0, T]L^\infty(d(x-y))L^2(d(x+y))}.
\end{align}
Thus, for any $\beta>0$, we  can choose $\sigma=\sigma(\beta)$ so that $1-\sigma\beta>0$.  
\end{remark}

\section{The TDHFB Equations}\label{Proof-of-main}
In this section we prove the local well-posedness of our system of nonlinear equations
addressed in the introduction.
First, let us write down the kernel form of the TDHFB equations
\begin{subequations}
\begin{align}
\left\{\frac{1}{i}\frac{\bd}{\bd t}-\lapl_{x_1}\right\}\varphi_t(x_1)=&\ -\int dy\ \{v_N(x_1-y)\rho_\Gamma(t, y)\}\cdot \varphi_t(x_1) \label{phi-eq}\\
          &\ - \int dy\ \{v_N(x_1-y)(\Gamma_t(y, x_1)-\overline{\varphi}_t(y)\varphi_t(x_1))\varphi_t(y)\} \nonumber \\
           &\ - \int dy\ \{v_N(x_1-y)(\Lambda_t(x_1, y)-\varphi_t(y)\varphi_t(x_1))\overline{\varphi}_t(y)\} \nonumber \\
\left\{ \frac{1}{i}\frac{\bd}{\bd t}-\lapl_{x_1}+\lapl_{x_2}\right\} \Gamma_t &(x_1, x_2) \label{gamma-eq} \\
=&\ - \int dy\ \{(v_N(x_1-y)-v_N(x_2-y))\overline{\Lambda}_t(x_1, y)\Lambda_t(y, x_2)\}\nonumber\\
&\ -\int dy\ \{(v_N(x_1-y)-v_N(x_2-y))\Gamma_t(x_1, y)\Gamma_t(y, x_2)\}\nonumber\\
&\ -\int dy\ \{(v_N(x_1-y)-v_N(x_2-y))\rho_\Gamma(t, y)\Gamma_t(x_1, x_2)\}\nonumber\\
&\ + 2\int dy\ \{(v_N(x_1-y)-v_N(x_2-y))|\varphi_t(y)|^2\overline{\varphi}_t(x_1)\varphi_t(x_2)\}\nonumber
\end{align}
\begin{align}
\bigg\{ \frac{1}{i}\frac{\bd}{\bd t}-\lapl_{x_1}-\lapl_{x_2}&+\frac{1}{N}v_N (x_1-x_2)\bigg\} \Lambda_t (x_1, x_2) \label{Lambda-eq}\\
=&\ - \int dy\ \{(v_N(x_1-y)+v_N(x_2-y))\rho_\Gamma(t, y)\Lambda_t(x_1, x_2)\}\nonumber\\
&\ -\int dy\ \{(v_N(x_1-y)+v_N(x_2\-y))\Lambda_t(x_1, y)\Gamma_t(y, x_2)\}\nonumber\\
&\ -\int dy\ \{(v_N(x_1-y)+v_N(x_2-y))\overline{\Gamma}_t(x_1, y)\Lambda_t(y, x_2)\}\nonumber\\
&\ + 2\int dy\ \{(v_N(x_1-y)+v_N(x_2-y))|\varphi_t(y)|^2\varphi_t(x_1)\varphi_t(x_2)\}\nonumber
\end{align}
\end{subequations}
The space $X_T$ is a Strichartz-type space equipped with a norm which is the sum of the following norms   
\begin{subequations}
\begin{align}
\vect{N}_T(\varphi) :=&\ \llp{\langle\grad_x\rangle^{\sigma}\varphi(t, x)}_{L^4[0,T]L^\infty(dx)}+\llp{\langle\grad_x\rangle^{\sigma}\varphi(t, x)}_{L^\infty[0, T]L^2(dx)}\\
\vect{N}_T(\Gamma):=&\  \llp{\langle\grad_{x, y}\rangle^{\sigma}\Gamma(t, x, y)}_{L^4[0, T]L^\infty(dx)L^2(dy)}\label{Gamma-norm}\\
&+ \llp{\langle\grad_{x, y}\rangle^{\sigma}\Gamma(t, x, y)}_{L^\infty[0, T]L^2(dxdy)}\nonumber\\
&+ \sup_{z}\llp{\grad^{1/2}_x\Gamma(t, x+z, x)}_{L^2(dtdx)}\ +\sup_z\llp{\grad^{1/2-\varepsilon}_x \Gamma(t, x+z, x)}_{L^2(dtdx)}\nonumber\\
\vect{N}_T(\Lambda):=&\ \llp{\langle \grad_{x, y}\rangle^{\sigma}\Lambda(t, x, y)}_{L^4[0, T]L^\infty(dx)L^2(dy)}\label{Lambda-norm}\\
&+\llp{\langle \grad_{x, y}\rangle^{\sigma}\Lambda(t, x, y)}_{L^\infty[0, T]L^2(dxdy)}
+\sup_z\llp{\langle\grad_x\rangle^{\sigma}
 \Lambda(t, x+z, x)}_{L^4[0, T]L^2(dx)}.\nonumber
\end{align}
\end{subequations} 
Moreover, let us denote the space of functions $(\varphi_t, \Gamma_t, \Lambda_t)$ where the above norms are finite for any $0\le T<\infty$  by $X_{\infty, \text{loc}}$. 

Let us present the main a-priori estimates of the article
\begin{thm}\label{main thm}
 Suppose $\varphi, \Gamma$ and $\Lambda$ solve (\ref{phi-eq}), (\ref{gamma-eq}) and (\ref{Lambda-eq}) respectively with
Schwartz initial condition $(\varphi_0, \Gamma_0, \Lambda_0)$. Then we have the following estimates
\begin{subequations}
\begin{align}
\vect{N}_T(\varphi) \lesssim&\ \llp{\langle \grad_{x}\rangle^{\sigma}\varphi_0}_{L^2(dx)}+T^\text{some power}\left(\vect{N}_T(\varphi)^2+\vect{N}_T(\Gamma)+\vect{N}_T(\Lambda)\right)\vect{N}_T(\varphi) \label{norm1}\\
 \vect{N}_T(\Gamma) \lesssim&\ \llp{\langle \grad_{x, y}\rangle^{\sigma}\Gamma_0}_{L^2(dxdy)}\label{norm2}
+T^\text{some power} \left(\vect{N}_T(\Gamma)^2+\vect{N}_T(\Lambda)^2+\vect{N}_T(\varphi)^4\right)\\
\vect{N}_T(\Lambda) \lesssim&\ \llp{\langle \grad_{x, y}\rangle^{\sigma}\Lambda_0}_{L^2(dxdy)}\label{norm3}
+T^\text{some power} \left(\vect{N}_T(\Gamma)\vect{N}_T(\Lambda)+\vect{N}_T(\varphi)^4\right). 
\end{align} 
\end{subequations}
In particular, there exists $T_0$ such that for all $T\leq T_0$ we have that
\begin{align*}
\vect{N}_T(X):=\vect{N}_T(\varphi)^2+\vect{N}_T(\Gamma)+\vect{N}_T(\Lambda) \lesssim\ 1.
\end{align*}
Similarly, the following estimates hold for the time derivative of $\varphi, \Gamma, \Lambda$, i.e.
\begin{subequations}
\begin{align}
\vect{N}_T(\bd_t\varphi) \lesssim&\ \Lp{\langle \grad_{x}\rangle^{\sigma}\bd_t\varphi\Big|_{t=0}}{L^2(dx)}\\
&+T^\text{some power}\left(\vect{N}_T(X)\vect{N}_T(\bd_t\varphi)+\vect{N}_T(\bd_tX)\vect{N}_T(\varphi)\right)\nonumber\\
\vect{N}_T(\bd_t\Gamma) \lesssim&\ \Lp{\langle \grad_{x, y}\rangle^{\sigma}\bd_t\Gamma\Big|_{t=0}}{L^2(dxdy)}\\
&+T^\text{some power} \vect{N}_T(X)\left(\vect{N}_T(\bd_tX)+\vect{N}_T(\varphi)\vect{N}_T(\bd_t\varphi)\right)\nonumber\\
\vect{N}_T(\bd_t\Lambda) \lesssim&\ \Lp{\langle \grad_{x, y}\rangle^{\sigma}\bd_t\Lambda\Big|_{t=0}}{L^2(dxdy)}\\
&+T^\text{some power} \vect{N}_T(X)\left(\vect{N}_T(\bd_tX)+\vect{N}_T(\varphi)\vect{N}_T(\bd_t\varphi)\right) \nonumber 
\end{align}
\end{subequations}
which again means there exists $T_0$ such that for all $T<T_0$ we have
\begin{align*}
 \vect{N}_T\left(\bd_tX\right):=\vect{N}_T\left(\bd_t\varphi\right)^2+\vect{N}_T\left(\bd_t\Gamma\right)+\vect{N}_T\left(\bd_t\Lambda\right) \lesssim 1.
\end{align*}
Indeed, for any $i, j \in \nn$ we have the estimates
\begin{align}\label{regular}
\vect{N}_T\left(\bd_t^i\grad_{x+y}^jX\right)\lesssim 1.  
\end{align}
\end{thm}

\begin{remark}
The reader should note that the solution obtained from the Banach fixed-point theorem is smooth if the initial data $(\varphi_0, \Gamma_0, \Lambda_0)$ is sufficiently smooth. Indeed, for each fixed $N$, one can show
\begin{align*}
\vect{N}_T\left(\bd_t^i\grad_{x}^j\grad_y^kX\right)\lesssim N^\text{some non-negative power}.
\end{align*}
Despite the fact that the higher Sobolev norms are not uniformly bounded in $N$, each of the solutions has sufficient smoothness for us to apply the conservation laws which we will state later in the section.
\end{remark}

\begin{remark}
The results of Proposition \ref{collapsing-est} and Proposition \ref{Lambda-collapse} for the homogeneous equations immediately
generalize to the cases $\Gamma(t, x+z, x)$ and $\Lambda(t, x+z, x)$, for 
any fixed $z$. For instance, in the case of the homogeneous $\Gamma$ equation, we see that
\begin{align*}
[\Gamma(t, x+z, x)]\ \widetilde{} = \frac{1}{4\pi |\xi|} \widehat{\Gamma_0}\left(\frac{\xi^2-\tau}{2\xi}, \frac{\xi^2+\tau}{2\xi} \right) e^{iz(\frac{\xi^2-\tau}{2\xi})}
\end{align*}
which, in norm, yields the same estimate as in Proposition \ref{collapsing-est}. As a consequence, all the
estimates  in \textsection 4,7 for the inhomogeneous equations will also hold
for  $\Gamma(t, x+z, x)$ and $\Lambda(t, x+z, x)$. This will allow us to close estimates pertaining to the shifted-diagonal quantities in the above norms. 
\end{remark}

We split the presentation of the proof of the theorem into two subsections.

\subsection{Proofs of Estimates (\ref{norm2}) and (\ref{norm3})}
Let us first consider equation (\ref{gamma-eq}). Since the term $(v_N\ast \rho_\Gamma)\cdot\Gamma$ has already been handled in \textsection 5, 
it suffices to consider only the terms $(v_N\bar\Lambda)\circ \Lambda$ 
and $(v_N\Gamma)\circ\Gamma$. 
In particular, it suffices to consider just the derivative of the terms since any computation
for the derivatives will encompass the computation for the non-derivative terms. 

Let us first handle the term $(v_N\ast\rho_\Gamma)\cdot \Gamma$. By a direct change of variables, we can rewrite the kernel composition as follows
\begin{align*}
(v_N\Gamma\circ \Gamma)(x, y) = \int dw\ v_N(x-w)\Gamma(x, w)\Gamma(w, y) = \int dz\ v_N(z) \Gamma(x, x-z)\Gamma(x-z, y).
\end{align*}
Then by Kato-Ponce inequality we obtain the following
\begin{align*}
\llp{\grad^{\sigma}_{x,y}&[(v_N \Gamma)\circ \Gamma ]}_{L^1[0, T]L^2(dxdy)}\\
\leq&\  \int dz\ |v_N(z)| \llp{\grad^{\sigma}_{x,y} [\Gamma(x, x-z)\Gamma(x-z, y)]}_{L^1[0, T]L^2(dxdy)}\\
\lesssim&\  \int dz\ |v_N(z)|  \llp{\grad^{\sigma}_x\Gamma(x, x-z)}_{L^{2}[0, T]L^{\widetilde p}(dx)} \llp{\Gamma(x, y)}_{L^2[0, T]L^{\widetilde r}(dx)L^2(dy)}\\
&+   \int dz\ |v_N(z)|  \llp{\Gamma(x, x-z)}_{L^{2}[0, T]L^p(dx)} \llp{\grad^{\sigma}_{x}\Gamma(x, y)}_{L^2[0, T]L^{r}(dx)L^2(dy)}
\end{align*}
where $p, r, \widetilde p, \widetilde r$ are the values stated in Remark \ref{pqr}. Applying Cauchy-Schwarz inequality in the time variable gives us
\begin{align*}
\llp{\grad^{\sigma}_{x,y}&[(v_N \Gamma)\circ \Gamma ]}_{L^1[0, T]L^2(dxdy)}\\
\lesssim&\  T^\text{some power}\int dz\ |v_N(z)|  \llp{\grad^{\frac{1}{2}-\varepsilon}_{x}\Gamma(x, x-z)}_{L^{2}(dtdx)}  \llp{\Gamma(x, y)}_{L^{\widetilde q}(dt)L^{\widetilde r}(dx)L^2(dy)}\\
&+   T^\text{some power}\int dz\ |v_N(z)|  \llp{\grad^{\frac{1}{2}-\varepsilon}_x\Gamma(x, x-z)}_{L^{2}(dtdx)} \llp{\grad^{\sigma}_{x,y}\Gamma(x, y)}_{L^q(dt)L^{r}(dx)L^2(dy)}\\
\lesssim&\ T^\text{some power}\vect{N}_T(\Gamma)^2
\end{align*}
where $(q, r)$ and $(\widetilde q, \widetilde r)$ are admissible pairs. Likewise, we have
\begin{align*}
\llp{\grad^{\sigma}_{x,y}&[(v_N \bar \Lambda)\circ \Lambda ]}_{L^1[0, T]L^2(dxdy)}\\
\leq&\ \int dz\ |v_N(z)| \llp{\grad_{x, y}^{\sigma} [\bar\Lambda(x, x-z)\Lambda(x-z, y)]}_{L^1[0, T]L^2(dxdy)}\\
\lesssim&\ \int dz\ |v_N(z)| \llp{\grad_x^{\sigma}\Lambda(x, x-z)}_{L^4[0, T]L^2(dx)} \llp{\Lambda(x, y)}_{L^{4/3}[0, T]L^\infty(dx)L^2(dy)}\\
&+ \int dz\ |v_N(z)| \llp{\Lambda(x, x-z)}_{L^4[0, T]L^2(dx)}\llp{\grad_{x}^{\sigma} \Lambda(x, y)}_{L^{4/3}[0, T]L^\infty(dx)L^2(dy)}\\
\lesssim&\ T^{1/2}\vect{N}_T(\Lambda)^2.
\end{align*}

As for equation (\ref{Lambda-eq}), there are essentially three terms we need to estimate,
namely $(v_N\ast \rho_\Gamma)\Lambda, (v_N\Gamma)\circ\Lambda$ and $(v_N\Lambda)\circ\Gamma$. 
Similar to the handling of the nonlinear terms for the $\Gamma$ equation, 
it suffices to look at just the derivatives of the nonlinear terms. 

For the first term, observe we have
\begin{align*}
\llp{\grad^{\sigma}_{x+y}[(v_N\ast\rho_\Gamma)\Lambda]&}_{L^{2}[0, T]L^{1+}(d(x-y))L^2(d(x+y))}\\
\lesssim&\ \llp{v_N\ast\rho_\Gamma}_{L^{2}[0, T]L^{2+}(dx)}\llp{\grad^{\sigma}_{x+y}\Lambda}_{L^{\infty}[0, T]L^{2}(d(x-y))L^2(d(x+y))}\\
&+\ \llp{v_N\ast\grad^{\sigma}_x\rho_\Gamma}_{L^{2}[0, T]L^{2+}(dx)}\llp{\Lambda}_{L^{\infty}[0, T]L^{2}(d(x-y))L^2(d(x+y))}\\
\lesssim&\ \llp{\grad^{\frac{1}{2}-\varepsilon}_x \rho_\Gamma}
_{L^{2}(dtdx)}\llp{\grad^{\sigma}_{x+y}
\Lambda}_{L^{\infty}[0, T]L^{2}(d(x-y))L^2(d(x+y))}\\
&+ \llp{\grad^{\frac{1}{2}-\varepsilon}_x \rho_\Gamma}_{L^{2}(dtdx)}
\llp{\Lambda}_{L^{\infty}[0, T]L^{2}(d(x-y))L^2(d(x+y))}\\
\lesssim&\ \vect{N}_T(\Gamma)\vect{N}_T(\Lambda)
\end{align*}
The terms $F=(v_N\Lambda)\circ\Gamma$ and $\Lambda\circ(v_N\Gamma)$ are handled similarly.

\subsection{Proof of Estimate (\ref{norm1})}
Let us begin by stating the following Strichartz estimate
\begin{prop}
Suppose $\varphi$ is a solution to $\vect{S}\varphi = F$ with initial condition $\varphi_0$ and let $(k, \ell)$ be an admissible pair. Then  it follows for all $\alpha>0$ we have
\begin{align}
\llp{\grad_x^\alpha\varphi}_{L^k[0, T]L^\ell(dx)} \lesssim \llp{\grad_x^\alpha\varphi_0}_{L^2(dx)}+\llp{\grad_x^\alpha F}_{L^{4/3}[0, T]L^{1}(dx)}.
\end{align}
\end{prop}

It suffice to consider only $(v_N\ast\rho_\Gamma)\cdot \varphi$ and $(v_N\Lambda)\circ \varphi$ since
the method applies word-for-word to the remaining nonlinear terms. 

For the first nonlinearity, we apply Kato-Ponce inequality to get the estimate
\begin{align*}
\llp{\grad_x^{\sigma}[(v_N\ast\rho_\Gamma)\cdot \varphi]}_{L^{4/3}[0, T]L^1(dx)} \lesssim&\  \llp{v_N\ast\grad_x^{\sigma}\rho_\Gamma}_{L^{4/3}[0, T]L^2(dx)}\llp{\varphi}_{L^{\infty}[0, T]L^2(dx)}\\
&\ + \llp{v_N\ast\rho_\Gamma}_{L^{4/3}[0, T]L^2(dx)}\llp{\grad_x^{\sigma}\varphi}_{L^{\infty}[0, T]L^2(dx)}\\
\lesssim&\ T^\text{some power}\vect{N}_T(\Gamma)\vect{N}_T(\varphi).
\end{align*}

For the second nonlinear term, we have
\begin{align*}
\llp{\grad_x^{\sigma}&[(v_N\Lambda)\circ \varphi]}_{L^{4/3}[0, T]L^1(dx)} \\
\lesssim& \int dz\ |v_N(z)|\llp{\grad_x^{\sigma}\Lambda(x, x-z)}_{L^{4/3}[0, T]L^2(dx)}\llp{\varphi(x-z)}_{L^\infty[0, T]L^2(dx)}\\
&+\ \int dz\ |v_N(z)|\llp{\Lambda(x, x-z)}_{L^{4/3}[0, T]L^2(dx)}\llp{\grad_x^{\sigma}\varphi(x-z)}_{L^\infty[0, T]L^2(dx)}\\
\lesssim&\ T^\text{some power}\vect{N}_T(\Lambda)\vect{N}_T(\varphi).
\end{align*}

\subsection{Global Well-Posedness of the TDHFB Equations}
In this subsection, we prove the global well-posedness of the TDHFB equations. Let us begin by
recalling the number and energy conservation laws derived in \textsection 9 of 
\cite{GrillakisMachedon2015}\footnote{cf. Corollary 2.7. and Theorem 2.8 in \cite{BaBreCFSig}}. Recall the total particle number is given by
\begin{align}
 \mathcal{N}:= N\int dx\ \rho_\Gamma(t, x) 
\end{align}
and the energy is defined by
\begin{align}
 \mathcal{E}& :=\ N\bigg\{\int dx\ |\grad_x\varphi(x)|^2+ \frac{1}{2N}\int dxdy\ |\grad_{x,y}\sh(k)(x,y)|^2  \\
&+ \frac{1}{2N}\int dxdydz\ v_N(x-y)|\varphi(x)\sh(k)(y, z)+\varphi(y)\sh(k)(x, z)|^2\nonumber\\
&+ \frac{1}{4}\int dxdy\ v_N(x-y)\left\{ 2|\Lambda(x, y)|^2+|\Gamma(x, y)|^2+\Gamma(x, x)\Gamma(y, y) \right\}\bigg\}  \nonumber
\end{align}
Note that we have suppress the dependence on $t$ in $\mathcal{E}$ for the sake of compactness of notation.

\begin{thm}[Conservation Laws]\label{conservation}
Suppose $(\varphi_t, \Gamma_t, \Lambda_t)$ solves the TDHFB equations and $v \in L^1(\rr)\cap C^\infty(\rr)$. Then the total particle number and energy is conserved.
\end{thm}
\begin{proof}
See \textsection 8 in \cite{GrillakisMachedon2015}.
\end{proof}

As an immediate corollary of Theorem \ref{conservation}, we have
\begin{cor}
 Let  $(\varphi_t, \Gamma_t, \Lambda_t)$ be a solution to the TDHFB equations. Then there exists a constant $C>0$ such that for any $T>0$ and $0<s<1$ we have that
\begin{align}
 \sup_{t\in [0, T]} \llp{(\varphi_t,\Gamma_t, \Lambda_t)}_{\mathcal{X}^s} \leq C,
\end{align}
independent of $N$. 
\end{cor}
\begin{proof}
 The estimate for $\varphi_t$ follows immedately by interpolating between the conservation of total particle number 
and conservation of energy. 
Next, applying Cauchy-Schwarz and the conservation of total particle number, we obtain the estimate
\begin{align}\label{con-est}
 \llp{\Gamma(t, \cdot)}_{L^2(dxdy)}\leq \llp{\varphi_t}_{L^2(dxdy)}^2
+\frac{1}{N}\llp{\sh(k_t)}_{L^2(dx)}^2 \lesssim 1.
\end{align}
Similarly, using Cauchy-Schwarz and the conservation of energy, we obtain
\begin{align}\label{con-est2}
 &\llp{\grad_x \Gamma(t, \cdot)}_{L^2(dxdy)} \\
&\leq \llp{\varphi_t}_{L^2(dx)}\llp{\grad_x\varphi_t}_{L^2(dx)}+\frac{1}{N}\llp{\sh(k_t)}_{L^2(dxdy)}
\llp{\grad_x\sh(k_t)}_{L^2(dxdy)}\lesssim 1\nonumber.
\end{align}
Interpolating (\ref{con-est}) and (\ref{con-est2}) yields a desired bound for $\Gamma_t$.

To uniformly bound $\Lambda_t$, we use the trig identity (\ref{trig-1}) to get the estimate
\begin{align}
\llp{\Lambda(t, \cdot)}_{L^2(dxdy)} \leq&\ \llp{\varphi_t}^2_{L^2(dx)}+\frac{1}{N}\llp{\sh(k_t)}_{L^2(dxdy)}\\
&\ +\frac{1}{N}\llp{\sh(k_t)}_{L^2(dxdy)}\llp{p(k_t)}_{L^2(dxdy)}.\nonumber
\end{align}
By identity (\ref{trig-2}), we see that $p\circ p+2p = \overline{\sh}\circ \sh$ which means
\begin{align*}
 \llp{p(k)}^2_{L^2(dx)}\leq \llp{p\circ p+2p}_{\Tr}  =\llp{\sh(k)}^2_{L^2(dxdy)}
\end{align*}
since $p(k)(x, x)\geq 0$. Hence by the conservation of total particle number we have that 
\begin{align*}
\llp{\Lambda(t, \cdot)}_{L^2(dxdy)} \lesssim 1.
\end{align*}
Similarly, we can show that $\llp{\grad_x \Lambda(t, \cdot)}_{L^2(dxdy)} \lesssim 1$. 
\end{proof}

\newcommand{\etalchar}[1]{$^{#1}$}
\providecommand{\bysame}{\leavevmode\hbox to3em{\hrulefill}\thinspace}
\providecommand{\MR}{\relax\ifhmode\unskip\space\fi MR }
\providecommand{\MRhref}[2]{%
  \href{http://www.ams.org/mathscinet-getitem?mr=#1}{#2}
}
\providecommand{\href}[2]{#2}

\end{document}